\documentclass[a4paper,twoside,10pt]{amsart}
\usepackage{amsmath,amsfonts,amssymb,amsthm}
\newtheorem{theorem}{Theorem}[section]
\newtheorem{lemma}{Lemma}[section]
\newtheorem{proposition}{Proposition}[section]

\usepackage{mathrsfs}
\usepackage{color,graphicx}
\usepackage[a4paper]{geometry}
\numberwithin{equation}{section}
\author[G. Nemes]{Gerg\H{o} Nemes}
\address{School of Mathematics, The University of Edinburgh, James Clerk Maxwell Building, The King's Buildings, Peter Guthrie Tait Road, Edinburgh EH9 3FD, UK}
\email{Gergo.Nemes@ed.ac.uk}


\begin{document}

\title[The large-argument asymptotics of the Hankel and Bessel functions]{Error bounds for the large-argument asymptotic expansions of the Hankel and Bessel functions}

\begin{abstract} In this paper, we reconsider the large-argument asymptotic expansions of the Hankel, Bessel and modified Bessel functions and their derivatives. New integral representations for the remainder terms of these asymptotic expansions are found and used to obtain sharp and realistic error bounds. We also give re-expansions for these remainder terms and provide their error estimates. A detailed discussion on the sharpness of our error bounds and their relation to other results in the literature is given. The techniques used in this paper should also generalize to asymptotic expansions which arise from an application of the method of steepest descents.
\end{abstract}
\maketitle

\section{Introduction and main results}

The large-$z$ asymptotic expansions of the Hankel functions $H_\nu^{\left(1\right)}\left(z\right)$ and $H_\nu^{\left(2\right)}\left(z\right)$, the Bessel functions $J_\nu\left(z\right)$ and $Y_\nu\left(z\right)$, and the modified Bessel functions $K_\nu\left(z\right)$ and $I_\nu\left(z\right)$ have a long and rich history. The earliest known attempt to obtain an asymptotic expansion is that of Poisson \cite{Poisson} in 1823 where the special case of $J_0\left(x\right)$ with $x$ being positive was considered. He derived the first few terms of an asymptotic expansion based on the differential equation satisfied by $J_0\left(x\right)$, but investigated neither the coefficients nor the remainder term. The first complete asymptotic expansion for $J_0\left(x\right)$, with a formula for the general term, was given by Hamilton \cite{Hamilton} in 1843. A similar (formal) analysis of $J_1\left(x\right)$ is due to Hansen \cite[pp. 119--123]{Hansen} from 1843. The complete asymptotic expansion of $J_n\left(x\right)$ for arbitrary integer order was first given by Jacobi \cite{Jacobi} in 1849. A rigorous treatment of Poisson's expansion was provided by Lipschitz \cite{Lipschitz} in 1859 with the aid of contour integration; here Jacobi's result was also studied briefly. The general asymptotic expansions of $J_\nu\left(z\right)$ and $Y_\nu\left(z\right)$, with a fixed complex $\nu$ and large complex $z$, were obtained (rigorously) by Hankel \cite{Hankel} in his memoir written in 1868. He also gave the corresponding expansions of the Hankel functions $H_\nu^{\left(1\right)}\left(z\right)$ and $H_\nu^{\left(2\right)}\left(z\right)$. The asymptotic expansion of the modified Bessel function $K_\nu\left(z\right)$ was established by Kummer \cite{Kummer} in 1837; this result was reproduced, with the addition of the corresponding formula for $I_\nu\left(z\right)$, by Kirchhoff \cite{Kirchhoff} in 1854. Estimates for the remainder terms of the asymptotic expansions of the Hankel and Bessel functions have been obtained by several authors in the past one and half century. We mention the works of Schl\"{a}fli \cite{Schlafli} from 1875, Watson \cite[pp. 209--210]{Watson} from 1922 and Meijer \cite{Meijer} from 1932, where the methods used were based on integral representations. A different approach to the problem, using differential equation techniques, was taken by Weber \cite{Weber} in 1890 and much later, in 1964, by Olver \cite{Olver1}. For a more detailed historical account, the reader is referred to Watson's monumental treatise on the theory of Bessel functions \cite[pp. 194--196]{Watson}.

In this paper, we reconsider the large-argument asymptotic expansions of the Hankel, Bessel and modified Bessel functions, and the corresponding results for their derivatives. In modern notation, the expansions for the Hankel functions, Bessel functions, and that of their derivatives may be written
\begin{equation}\label{eq94}
H_\nu ^{\left( 1 \right)} \left( z \right) \sim \left( {\frac{2}{\pi z}} \right)^{\frac{1}{2}} e^{i\omega } \sum\limits_{n = 0}^{\infty} {i^n \frac{{a_n \left( \nu  \right)}}{z^n}}, 
\end{equation}
\begin{equation}\label{eq060}
H_\nu ^{\left( 1 \right)\prime} \left( z \right) \sim i\left( {\frac{2}{\pi z}} \right)^{\frac{1}{2}} e^{i\omega } \sum\limits_{n = 0}^{\infty} {i^n \frac{{b_n \left( \nu  \right)}}{z^n}}, 
\end{equation}
as $z\to \infty$ in the sector $- \pi + \delta  \le \arg z \le 2\pi - \delta$;
\begin{equation}\label{eq95}
H_\nu ^{\left( 2 \right)} \left( z \right) \sim \left( {\frac{2}{{\pi z}}} \right)^{\frac{1}{2}} e^{ - i\omega } \sum\limits_{n = 0}^{\infty} {\left( { - i} \right)^n \frac{{a_n \left( \nu  \right)}}{{z^n }}},
\end{equation}
\begin{equation}\label{eq061}
H_\nu ^{\left( 2 \right)\prime} \left( z \right) \sim -i\left( {\frac{2}{{\pi z}}} \right)^{\frac{1}{2}} e^{ - i\omega } \sum\limits_{n = 0}^{\infty} {\left( { - i} \right)^n \frac{{b_n \left( \nu  \right)}}{{z^n }}},
\end{equation}
as $z\to \infty$ in the sector $- 2\pi + \delta  \le \arg z \le \pi - \delta$;
\begin{equation}\label{eq86}
J_\nu  \left( z \right)\sim \left( {\frac{2}{{\pi z}}} \right)^{\frac{1}{2}} \left( {\cos \omega \sum\limits_{n = 0}^{\infty} {\left( { - 1} \right)^n \frac{{a_{2n} \left( \nu  \right)}}{{z^{2n} }}} - \sin \omega \sum\limits_{m = 0}^{\infty} {\left( { - 1} \right)^m \frac{{a_{2m + 1} \left( \nu  \right)}}{{z^{2m + 1} }}}  } \right),
\end{equation}
\begin{equation}\label{eq888}
J'_\nu  \left( z \right)\sim -\left( {\frac{2}{{\pi z}}} \right)^{\frac{1}{2}} \left( {\sin \omega \sum\limits_{n = 0}^{\infty} {\left( { - 1} \right)^n \frac{{b_{2n} \left( \nu  \right)}}{{z^{2n} }}} +\cos \omega \sum\limits_{m = 0}^{\infty} {\left( { - 1} \right)^m \frac{{b_{2m + 1} \left( \nu  \right)}}{{z^{2m + 1} }}}  } \right),
\end{equation}
\begin{equation}\label{eq80}
Y_\nu  \left( z \right) \sim \left( {\frac{2}{{\pi z}}} \right)^{\frac{1}{2}} \left( {\sin \omega \sum\limits_{n = 0}^{\infty} {\left( { - 1} \right)^n \frac{{a_{2n} \left( \nu  \right)}}{z^{2n}}} + \cos \omega \sum\limits_{m = 0}^{\infty} {\left( { - 1} \right)^m \frac{{a_{2m + 1} \left( \nu  \right)}}{{z^{2m + 1} }}}  } \right)
\end{equation}
and
\begin{equation}\label{eq108}
Y'_\nu  \left( z \right) \sim \left( {\frac{2}{{\pi z}}} \right)^{\frac{1}{2}} \left( {\cos \omega \sum\limits_{n = 0}^{\infty} {\left( { - 1} \right)^n \frac{{b_{2n} \left( \nu  \right)}}{z^{2n}}} -\sin \omega \sum\limits_{m = 0}^{\infty} {\left( { - 1} \right)^m \frac{{b_{2m + 1} \left( \nu  \right)}}{{z^{2m + 1} }}}  } \right),
\end{equation}
as $z\to \infty$ in the sector $\left| {\arg z} \right| \le \pi - \delta$, with $\delta$ being an arbitrary small (fixed) positive constant and $\omega  = z - \frac{\pi }{2}\nu  - \frac{\pi }{4}$ (see, e.g., \cite[\S 10.17(i) and \S 10.17(ii)]{NIST}). The branch of the square root in these expansions is determined by $z^{\frac{1}{2}}  = \exp \left( {\frac{1}{2}\log \left| z \right| + \frac{i}{2}\arg z} \right)$. The coefficients $a_n \left( \nu  \right)$ and $b_n \left( \nu  \right)$ are polynomials in $\nu^2$ of degree $n$; their explicit forms are as follows:
\begin{gather}\label{eq612}
\begin{split}
a_n \left( \nu  \right) & = \left( { - 1} \right)^n \frac{{\cos \left( {\pi \nu } \right)}}{\pi }\frac{{\Gamma \left( {n + \frac{1}{2} + \nu } \right)\Gamma \left( {n + \frac{1}{2} - \nu } \right)}}{{2^n \Gamma \left( {n + 1} \right)}} \\ & = \frac{(4\nu ^2  - 1^2 )( 4\nu ^2  - 3^2 ) \cdots ( 4\nu ^2  - \left( {2n - 1} \right)^2  )}{8^n \Gamma \left( {n + 1} \right)},
\end{split}
\end{gather}
for $n\geq 0$ and
\begin{align*}
b_0\left( \nu  \right)=1,\;\, b_n \left( \nu  \right) & = \left( { - 1} \right)^{n + 1} \frac{{\cos \left( {\pi \nu } \right)}}{\pi }\frac{{\Gamma \left( {n - \frac{1}{2} + \nu } \right)\Gamma \left( {n - \frac{1}{2} - \nu } \right)(4\nu ^2  + 4n^2  - 1)}}{{2^{n + 2} \Gamma \left( {n + 1} \right)}} \\ & = \frac{{( (4\nu ^2  - 1^2 )(4\nu ^2  - 3^2 ) \cdots (4\nu ^2  - \left( {2n - 3} \right)^2 ) )(4\nu ^2  + 4n^2  - 1)}}{{8^n \Gamma \left( {n + 1} \right)}},
\end{align*}
for $n\geq 1$ (with the empty products interpreted as $1$). The expansions \eqref{eq94}, \eqref{eq95}, \eqref{eq86} and \eqref{eq80} are due to Hankel. The results established by Poisson, Hansen and Jacobi are all special cases of the asymptotic expansion \eqref{eq86}. If $2\nu$ equals an odd integer, then the right-hand sides of \eqref{eq94}--\eqref{eq108} terminate and represent the corresponding function exactly.

The analogous expansions for the modified Bessel functions and their derivatives are
\begin{equation}\label{eq88}
K_\nu  \left( z \right) \sim \left( {\frac{\pi }{{2z}}} \right)^{\frac{1}{2}} e^{ - z} \sum\limits_{n = 0}^\infty  {\frac{{a_n \left( \nu  \right)}}{{z^n }}} ,
\end{equation}
\begin{equation}\label{eq222}
K'_\nu  \left( z \right) \sim -\left( {\frac{\pi }{{2z}}} \right)^{\frac{1}{2}} e^{ - z} \sum\limits_{n = 0}^\infty  {\frac{{b_n \left( \nu  \right)}}{{z^n }}} ,
\end{equation}
as $z\to \infty$ in the sector $\left| {\arg z} \right| \le \frac{3\pi}{2} - \delta$; and
\begin{equation}\label{eq87}
I_\nu  \left( z \right) \sim \frac{{e^z }}{{\left( {2\pi z} \right)^{\frac{1}{2}} }}\sum\limits_{n = 0}^\infty  {\left( { - 1} \right)^n \frac{{a_n \left( \nu  \right)}}{{z^n }}}  \pm i e^{ \pm \pi i\nu } \frac{{e^{ - z} }}{{\left( {2\pi z} \right)^{\frac{1}{2}} }}\sum\limits_{m = 0}^\infty  {\frac{{a_m \left( \nu  \right)}}{z^m }} ,
\end{equation}
\begin{equation}\label{eq223}
I'_\nu  \left( z \right) \sim \frac{{e^z }}{{\left( {2\pi z} \right)^{\frac{1}{2}} }}\sum\limits_{n = 0}^\infty  {\left( { - 1} \right)^n \frac{{b_n \left( \nu  \right)}}{{z^n }}}  \mp i e^{ \pm \pi i\nu } \frac{{e^{ - z} }}{{\left( {2\pi z} \right)^{\frac{1}{2}} }}\sum\limits_{m = 0}^\infty  {\frac{{b_m \left( \nu  \right)}}{z^m }} ,
\end{equation}
as $z\to \infty$ in the sectors $-\frac{\pi}{2}+\delta \leq \pm \arg z \leq \frac{3\pi}{2} - \delta$, with $\delta$ being an arbitrary small (fixed) positive constant (see, for instance, \cite[\S 10.40(i)]{NIST}). As before, the square roots are determined by $z^{\frac{1}{2}}  = \exp \left( {\frac{1}{2}\log \left| z \right| + \frac{i}{2}\arg z} \right)$. The original result of Kirchhoff omits the second component of the asymptotic expansion \eqref{eq87}, which is permitted if we restrict $z$ to the smaller sector $\left| {\arg z} \right| \le \frac{\pi}{2} - \delta$. (In this sector, the second component is exponentially small compared to any of the terms in the first component for large $z$ and is therefore negligible.) The same observation applies for \eqref{eq223}. The expansions \eqref{eq88}--\eqref{eq223} terminate and are exact when $2\nu$ equals an odd integer.

It is important to note that these asymptotic expansions are not uniform with respect to $\nu$; we have to require $\nu^2  = o(\left| z \right|)$ in order to satisfy Poincar\'e's definition. There exist other types of large-$z$ expansions which are valid under the weaker condition $\nu  = o(\left| z \right|)$ (see, for instance, \cite{Bickley}, \cite{Goldstein} or \cite{Heitman}); however these expansions do not lend themselves to treatment with our methods.

The main aim of the present paper is to derive new error bounds for the asymptotic expansions \eqref{eq94}--\eqref{eq223} and their re-expanded versions. Due to the various relations between the Hankel and Bessel functions, it is enough to study the remainder terms of the asymptotic expansions \eqref{eq86}, \eqref{eq888}, \eqref{eq88} and \eqref{eq222} (see Appendix \ref{appendixa} for more details). Thus, for any non-negative integers $N$ and $M$, we write
\begin{equation}\label{eq70}
K_\nu  \left( z \right) = \left( {\frac{\pi }{{2z}}} \right)^{\frac{1}{2}} e^{ - z} \left( {\sum\limits_{n = 0}^{N - 1} {\frac{{a_n \left( \nu  \right)}}{{z^n }}}  + R_N^{\left( K \right)} \left( {z,\nu } \right)} \right),
\end{equation}
\begin{equation}\label{eq843}
K'_\nu  \left( z \right) = -\left( {\frac{\pi }{{2z}}} \right)^{\frac{1}{2}} e^{ - z} \left( {\sum\limits_{n = 0}^{N - 1} {\frac{{b_n \left( \nu  \right)}}{{z^n }}}  + R_N^{( K')} \left( {z,\nu } \right)} \right),
\end{equation}
\begin{multline*}
J_\nu  \left( z \right) = \left( {\frac{2}{{\pi z}}} \right)^{\frac{1}{2}} \left( \cos \omega \left( {\sum\limits_{n = 0}^{N - 1} {\left( { - 1} \right)^n \frac{{a_{2n} \left( \nu  \right)}}{{z^{2n} }}}  + R_{2N}^{\left( J \right)} \left( {z,\nu } \right)} \right) \right. \\ \left.- \sin \omega \left( {\sum\limits_{m = 0}^{M - 1} {\left( { - 1} \right)^m \frac{{a_{2m + 1} \left( \nu  \right)}}{{z^{2m + 1} }}}  - R_{2M + 1}^{\left( J \right)} \left( {z,\nu } \right)} \right) \right)
\end{multline*}
and
\begin{multline*}
J'_\nu  \left( z \right) = -\left( {\frac{2}{{\pi z}}} \right)^{\frac{1}{2}} \left( \sin \omega \left( {\sum\limits_{n = 0}^{N - 1} {\left( { - 1} \right)^n \frac{{b_{2n} \left( \nu  \right)}}{{z^{2n} }}}  + R_{2N}^{(J')} \left( {z,\nu } \right)} \right) \right. \\ \left.+ \cos \omega \left( {\sum\limits_{m = 0}^{M - 1} {\left( { - 1} \right)^m \frac{{b_{2m + 1} \left( \nu  \right)}}{{z^{2m + 1} }}}  - R_{2M + 1}^{(J')} \left( {z,\nu } \right)} \right) \right).
\end{multline*}
Throughout this paper, if not stated otherwise, empty sums are taken to be zero. The derivations of the estimates for $R_N^{\left( K \right)} \left( {z,\nu } \right)$, $R_N^{(K')} \left( {z,\nu } \right)$, $R_{2N}^{\left( J \right)} \left( {z,\nu } \right)$, $R_{2M+1}^{\left( J \right)} \left( {z,\nu } \right)$, $R_{2N}^{(J')} \left( {z,\nu } \right)$ and $R_{2M+1}^{(J')} \left( {z,\nu } \right)$ are based on new integral representations of these remainder terms. In expressing these new integral representations, we shall employ the following notation ($w \neq 0$):
\[
\Lambda _p \left( w \right) = w^p e^w \Gamma \left( { 1- p,w} \right) \; \text{ and } \; \Pi _p \left( w \right) = \frac{1}{2}\left( {\Lambda _p \left( {we^{\frac{\pi }{2}i} } \right) + \Lambda _p \left( {we^{ - \frac{\pi }{2}i} } \right)} \right),
\]
where $\Gamma \left( { 1- p,w} \right)$ is the incomplete gamma function. The functions $\Lambda _p \left( w \right)$ and $\Pi _p \left( w \right)$ were originally introduced by Dingle \cite{Dingle1}\cite[pp. 406--407 and p. 415]{Dingle2} and, following his convention, we refer to them as basic terminants (but note that Dingle's notation slightly differs from ours, e.g., $\Lambda_{p-1} \left( w \right)$ is used for our $\Lambda_p \left( w \right)$). These basic terminants are multivalued functions of their argument $w$ and, when the argument is fixed, are entire functions of their order $p$. We shall also use the concept of the regularized hypergeometric function ${\bf F}\left( {a,b;c;w} \right)$ which is defined by the power series expansion
\[
{\bf F}\left( {a,b;c;w} \right) = \frac{1}{{\Gamma \left( a \right)\Gamma \left( b \right)}}\sum\limits_{n = 0}^\infty  {\frac{{\Gamma \left( {a + n} \right)\Gamma \left( {b + n} \right)}}{{\Gamma \left( {c + n} \right)\Gamma \left( {n + 1} \right)}}w^n } 
\]
for $\left| w \right| < 1$ and by analytic continuation elsewhere \cite[\S 15.2]{NIST}. The parameters $a$, $b$ and $c$ of this function can take arbitrary complex values.

We are now in a position to formulate our main results. In Theorem \ref{theorem1}, we give new integral representations for the remainder terms $R_N^{\left( K \right)} \left( {z,\nu } \right)$ and $R_N^{\left( J \right)} \left( {z,\nu } \right)$ involving a free complex parameter $\lambda$, $\Re \left( \lambda  \right) > 0$.

\begin{theorem}\label{theorem1} Let $N$ be a non-negative integer and let $\nu$ and $\lambda$ be arbitrary complex numbers such that $\left| {\Re \left( \nu  \right)} \right| < N + \frac{1}{2}$ and $\Re \left( \lambda  \right) > 0$. Then
\begin{gather}\label{eq9}
\begin{split}
R_N^{\left( K \right)} \left( {z,\nu } \right) = \; & \left( { - 1} \right)^N \frac{{\cos \left( {\pi \nu } \right)}}{\pi }\frac{{\Gamma \left( {N + \lambda } \right)}}{{2^N}}\frac{1}{{z^N }} \\ & \times \int_0^{ + \infty } {\frac{{t^{\lambda  - 1} }}{{\left( {1 + t} \right)^{N + \lambda } }}{\bf F} \left( {\nu  + \frac{1}{2}, - \nu  + \frac{1}{2};\lambda ; - t} \right)\Lambda _{N + \lambda} \left( {2z\left( {1 + t} \right)} \right)dt} ,
\end{split}
\end{gather}
provided $\left| {\arg z} \right| < \frac{{3\pi }}{2}$, and
\begin{gather}\label{eq54}
\begin{split}
R_N^{\left( J \right)} \left( {z,\nu } \right) = \; & \left( { - 1} \right)^{\left\lfloor {N/2} \right\rfloor }  \frac{{\cos \left( {\pi \nu } \right)}}{\pi }\frac{{\Gamma \left( {N + \lambda } \right)}}{{2^N }}\frac{1}{{z^N }} \\ & \times \int_0^{ + \infty } {\frac{{t^{\lambda  - 1} }}{{\left( {1 + t} \right)^{N + \lambda } }}{\bf F}\left( {\nu  + \frac{1}{2}, - \nu  + \frac{1}{2};\lambda ; - t} \right)\Pi _{N + \lambda} \left( {2z\left( {1 + t} \right)} \right)dt} ,
\end{split}
\end{gather}
provided $\left| {\arg z} \right| < \pi$.
\end{theorem}

If $\lambda=0$ and $N\geq 1$, a slight modification of \eqref{eq9} and \eqref{eq54} holds.

\begin{theorem}\label{theorem2} Let $N$ be a positive integer and let $\nu$ be an arbitrary complex number such that $\left| {\Re \left( \nu  \right)} \right| < N + \frac{1}{2}$. Then
\begin{gather}\label{eq45}
\begin{split}
R_N^{\left( K \right)} \left( {z,\nu } \right) = \; & \left( { - 1} \right)^N \frac{{\cos \left( {\pi \nu } \right)}}{\pi }\frac{{\Gamma \left( N \right)}}{{2^N }}\frac{1}{{z^N }} \\ & \times  \left( {\Lambda _{N } \left( {2z} \right) + \int_0^{ + \infty } {\frac{{t^{ - 1} }}{{\left( {1 + t} \right)^N }}{\bf F}\left( {\nu  + \frac{1}{2}, - \nu  + \frac{1}{2};0; - t} \right)\Lambda _{N} \left( {2z\left( {1 + t} \right)} \right)dt} } \right),
\end{split}
\end{gather}
provided $\left| {\arg z} \right| < \frac{{3\pi }}{2}$, and
\begin{gather}\label{eq55}
\begin{split}
R_N^{\left( J \right)} \left( {z,\nu } \right) = \; & \left( { - 1} \right)^{\left\lfloor {N/2} \right\rfloor } \frac{{\cos \left( {\pi \nu } \right)}}{\pi }\frac{{\Gamma \left( N \right)}}{{2^N }}\frac{1}{{z^N }}  \\ & \times  \left( {\Pi_{N } \left( {2z} \right) + \int_0^{ + \infty } {\frac{{t^{ - 1} }}{{\left( {1 + t} \right)^N }}{\bf F}\left( {\nu  + \frac{1}{2}, - \nu  + \frac{1}{2};0; - t} \right)\Pi _{N} \left( {2z\left( {1 + t} \right)} \right)dt} } \right),
\end{split}
\end{gather}
provided $\left| {\arg z} \right| < \pi$.
\end{theorem}

Analogous expressions for the remainder terms $R_N^{(K')} \left( {z,\nu } \right)$ and $R_N^{(J')} \left( {z,\nu } \right)$ can be written down by applying Theorems \ref{theorem1} and \ref{theorem2} together with the functional equations
\begin{equation}\label{eq101}
2R_N^{(K')} \left( {z,\nu } \right) = R_N^{\left( K \right)} \left( {z,\nu  + 1} \right) + R_N^{\left( K \right)} \left( {z,\nu  - 1} \right)
\end{equation}
and
\begin{equation}\label{eq102}
2R_N^{(J')} \left( {z,\nu } \right) = R_N^{\left( J \right)} \left( {z,\nu  + 1} \right) + R_N^{\left( J \right)} \left( {z,\nu  - 1} \right).
\end{equation}
These functional equations follow directly from the connection formulae $- 2K'_\nu  \left( z \right) = K_{\nu + 1} \left( z \right) + K_{\nu  - 1} \left( z \right)$ and $-2J'_\nu  \left( z \right) = J_{\nu  + 1} \left( z \right) - J_{\nu  - 1} \left( z \right)$ (see, e.g., \cite[10.6.E1 and 10.29.E1]{NIST}).

It was shown by Boyd \cite{Boyd} that for any non-negative integer $N$, the remainder term $R_N^{\left( K \right)} \left( {z,\nu } \right)$ can be expressed as
\begin{equation}\label{eq1}
R_N^{\left( K \right)} \left( {z,\nu } \right) = \left( { - 1} \right)^N \left( {\frac{2}{\pi }} \right)^{\frac{1}{2}} \frac{{\cos \left( {\pi \nu } \right)}}{\pi }\frac{1}{{z^N }}\int_0^{ + \infty } {\frac{{t^{N - \frac{1}{2}} e^{ - t} }}{{1 + t/z}}K_\nu  \left( t \right)dt} ,
\end{equation}
provided that $\left| {\arg z} \right| < \pi$ and $\left| {\Re \left( \nu  \right)} \right| < N + \frac{1}{2}$. This representation of $R_N^{\left( K \right)} \left( {z,\nu } \right)$ is the central tool in Boyd's asymptotic analysis of the modified Bessel function $K_\nu\left(z\right)$, in particular, in the derivation of his bounds for $R_N^{\left( K \right)} \left( {z,\nu } \right)$. In the following two theorems, we present results for the remainders $R_N^{\left( J \right)} \left( {z,\nu } \right)$, $R_N^{(K')} \left( {z,\nu } \right)$ and $R_{N}^{(J')} \left( {z,\nu } \right)$ analogous to \eqref{eq1}. In our analysis, the purpose of these representations is not to find bounds for the remainders themselves but to prove error bounds for the re-expansions of these remainders.

\begin{theorem}\label{theorem8} Let $N$ be a non-negative integer and let $\nu$ be an arbitrary complex number such that $\left| {\Re \left( \nu  \right)} \right| < N + \frac{1}{2}$. Then
\begin{equation}\label{eq91}
R_N^{\left( J \right)} \left( {z,\nu } \right) = \left( { - 1} \right)^{\left\lfloor N/2\right\rfloor} \left( {\frac{2}{\pi }} \right)^{\frac{1}{2}} \frac{{\cos \left( {\pi \nu } \right)}}{\pi }\frac{1}{{z^N }}\int_0^{ + \infty } {\frac{{t^{N - \frac{1}{2}} e^{ - t} }}{{1 + \left(t/z\right)^2}}K_\nu  \left( t \right)dt} ,
\end{equation}
provided $\left| {\arg z} \right| < \frac{\pi}{2}$.
\end{theorem}

\begin{theorem}\label{theorem9} Let $N$ be a positive integer and let $\nu$ be an arbitrary complex number such that $\left| {\Re \left( \nu  \right)} \right| < N - \frac{1}{2}$. Then
\begin{equation}\label{eq92}
R_N^{(K')} \left( {z,\nu } \right) = \left( { - 1} \right)^N \left( {\frac{2}{\pi }} \right)^{\frac{1}{2}} \frac{{\cos \left( {\pi \nu } \right)}}{\pi }\frac{1}{{z^N }}\int_0^{ + \infty } {\frac{{t^{N - \frac{1}{2}} e^{ - t} }}{{1 + t/z}}K'_\nu  \left( t \right)dt} ,
\end{equation}
provided $\left| {\arg z} \right| < \pi$, and
\begin{equation}\label{eq93}
R_N^{(J')} \left( {z,\nu } \right) = \left( { - 1} \right)^{\left\lfloor N/2\right\rfloor} \left( {\frac{2}{\pi }} \right)^{\frac{1}{2}} \frac{{\cos \left( {\pi \nu } \right)}}{\pi }\frac{1}{{z^N }}\int_0^{ + \infty } {\frac{{t^{N - \frac{1}{2}} e^{ - t} }}{{1 + \left(t/z\right)^2}}K'_\nu  \left( t \right)dt} ,
\end{equation}
provided $\left| {\arg z} \right| < \frac{\pi}{2}$.
\end{theorem}

The subsequent two theorems provide bounds for the remainders $R_N^{\left( K \right)} \left( {z,\nu } \right)$, $R_N^{\left( J \right)} \left( {z,\nu } \right)$, $R_N^{(K')} \left( {z,\nu } \right)$ and $R_N^{(J')} \left( {z,\nu } \right)$ when $\nu$ is real. These error bounds may be further simplified by employing the various estimates for $\mathop {\sup }\nolimits_{r \geq 1} \big| {\Lambda _{p} \left( {2zr } \right)} \big|$ and $\mathop {\sup }\nolimits_{r \geq 1} \big| {\Pi _{p} \left( {2zr } \right)} \big|$ given in Appendix \ref{appendixb}. Note that those estimates in Appendix \ref{appendixb} that depend on the (positive) order $p$ are monotonically increasing functions of $p$. Therefore, in the following theorems, we minimize, as much as our methods allow, the order of the basic terminants with respect to the following natural requirement: the form of each bound is directly related to the first omitted term of the corresponding asymptotic expansion. We remark that the bounds \eqref{eq24}--\eqref{eq701} remain true even if the orders of the basic terminants are taken to be any positive quantity at least $N + \max \left( {0,\frac{1}{2} - \left| \nu  \right|} \right)$.

\begin{theorem}\label{theorem3} Let $N$ be a non-negative integer and let $\nu$ be an arbitrary real number such that $\left| {\nu} \right| < N + \frac{1}{2}$. Then
\begin{equation}\label{eq24}
\big| {R_N^{\left( K \right)} \left( {z,\nu } \right)} \big| \le \frac{{\left| {a_N \left( \nu  \right)} \right|}}{{\left| z \right|^N }}\mathop {\sup }\limits_{r \geq 1} \big| {\Lambda _{N + \max \left( {0,\frac{1}{2} - \left| \nu  \right|} \right)} \left( {2zr } \right)} \big|,
\end{equation}
provided $\left| {\arg z} \right| < \frac{{3\pi }}{2}$, and
\begin{equation}\label{eq25}
\big| {R_N^{\left( J \right)} \left( {z,\nu } \right)} \big| \le \frac{{\left| {a_N \left( \nu  \right)} \right|}}{{\left| z \right|^N }}\mathop {\sup }\limits_{r \geq 1} \big| {\Pi _{N + \max \left( {0,\frac{1}{2} - \left| \nu  \right|} \right)} \left( {2zr } \right)} \big|,
\end{equation}
provided $\left| {\arg z} \right| < \pi$.
\end{theorem}

\begin{theorem}\label{theorem4} Let $N$ be a positive integer and let $\nu$ be an arbitrary real number such that $\left| {\nu} \right| < N - \frac{1}{2}$. Then
\begin{equation}\label{eq700}
\big| {R_N^{(K')} \left( {z,\nu } \right)} \big| \le \frac{{\left| {b_N \left( \nu  \right)} \right|}}{{\left| z \right|^N }}\mathop {\sup }\limits_{r \geq 1} \big| {\Lambda _{N} \left( {2zr } \right)} \big|,
\end{equation}
provided $\left| {\arg z} \right| < \frac{{3\pi }}{2}$, and
\begin{equation}\label{eq701}
\big| {R_N^{(J')} \left( {z,\nu } \right)} \big| \le \frac{{\left| {b_N \left( \nu  \right)} \right|}}{{\left| z \right|^N }}\mathop {\sup }\limits_{r \geq 1} \big| {\Pi _{N} \left( {2zr } \right)} \big|,
\end{equation}
provided $\left| {\arg z} \right| < \pi$.
\end{theorem}

It is well known that in the special case when $z$ is positive and $\nu$ is real, $\left|\nu\right|<N+\frac{1}{2}$, we have
\begin{equation}\label{eq711}
R_N^{\left( K \right)} \left( {z,\nu } \right) = \frac{{a_N \left( \nu  \right)}}{{z^N }}\theta _N^{\left( K \right)} \left( {z,\nu } \right) \; \text{ and } \; R_N^{\left( J \right)} \left( {z,\nu } \right) = \left( { - 1} \right)^{\left\lceil {N/2} \right\rceil } \frac{{a_N \left( \nu  \right)}}{{z^N }}\theta _N^{\left( J \right)} \left( {z,\nu } \right).
\end{equation}
Here $0 < \theta _N^{\left( K \right)} \left( {z,\nu } \right),\theta _N^{\left( J \right)} \left( {z,\nu } \right) < 1$ are appropriate numbers that depend on $z, \nu$ and $N$ (see, for instance, \cite[\S 10.40(ii)]{NIST} or \cite[pp. 207 and 209]{Watson}). Said differently, the remainder terms $R_N^{\left( K \right)} \left( {z,\nu } \right)$ and $R_N^{\left( J \right)} \left( {z,\nu } \right)$ do not exceed the corresponding first neglected terms in absolute value and have the same sign provided that $z > 0$ and $\left|\nu\right|<N+\frac{1}{2}$. In the theorem below, we present the analogous results for the remainder terms $R_N^{(K')} \left( {z,\nu } \right)$ and $R_N^{(J')} \left( {z,\nu } \right)$.

\begin{theorem}\label{theorem5} Let $N$ be a positive integer, $z$ be a positive real number, and let $\nu$ be an arbitrary real number such that $\left| {\nu} \right| < N - \frac{1}{2}$. Then
\begin{equation}\label{eq321}
R_N^{(K')} \left( {z,\nu } \right) = \frac{{b_N \left( \nu  \right)}}{{z^N }}\theta _N^{(K')} \left( {z,\nu } \right) \; \text{ and } \; R_N^{(J')} \left( {z,\nu } \right) = \left( { - 1} \right)^{\left\lceil {N/2} \right\rceil } \frac{{b_N \left( \nu  \right)}}{{z^N }}\theta _N^{(J')} \left( {z,\nu } \right),
\end{equation}
where $0 < \theta _N^{(K')} \left( {z,\nu } \right),\theta _N^{(J')} \left( {z,\nu } \right) < 1$ are suitable numbers that depend on $z, \nu$ and $N$.
\end{theorem}

We now consider the case when $\nu$ is complex. In the following theorems, we have chosen $N+\frac{1}{2}$ as the order of the basic terminants because this is the value that allows us to express the bounds in a form closely related to the first omitted term in the corresponding asymptotic expansion. The precise relation to the first omitted term is examined in Section \ref{section5}.

\begin{theorem}\label{theorem6} Let $N$ be a non-negative integer and let $\nu$ be an arbitrary complex number such that $\left| {\Re \left( \nu  \right)} \right| < N + \frac{1}{2}$. Then
\begin{equation}\label{eq01}
\big| {R_N^{\left( K \right)} \left( {z,\nu } \right)} \big| \le \frac{{\left| {\cos \left( {\pi \nu } \right)} \right|}}{{\left| {\cos \left( {\pi \Re \left( \nu  \right)} \right)} \right|}}\frac{{\left| {a_N \left( {\Re \left( \nu  \right)} \right)} \right|}}{{\left| z \right|^N }}\mathop {\sup }\limits_{r \geq 1} \big| {\Lambda _{N + \frac{1}{2}} \left( {2zr } \right)} \big|,
\end{equation}
provided $\left| {\arg z} \right| < \frac{3\pi}{2}$, and
\begin{equation}\label{eq02}
\big| {R_N^{\left( J \right)} \left( {z,\nu } \right)} \big| \le \frac{{\left| {\cos \left( {\pi \nu } \right)} \right|}}{{\left| {\cos \left( {\pi \Re \left( \nu  \right)} \right)} \right|}}\frac{{\left| {a_N \left( {\Re \left( \nu  \right)} \right)} \right|}}{{\left| z \right|^N }}\mathop {\sup }\limits_{r \geq 1} \big| {\Pi _{N + \frac{1}{2}} \left( {2zr } \right)} \big|,
\end{equation}
provided $\left| {\arg z} \right| < \pi$. If $2\Re \left( \nu  \right)$ is an odd integer, then the limiting values have to be taken in these bounds.
\end{theorem}

\begin{theorem}\label{theorem7} Let $N$ be a positive integer and let $\nu$ be an arbitrary complex number such that $\left| {\Re \left( \nu  \right)} \right| < N - \frac{1}{2}$. Then
\begin{equation}\label{eq220}
\big| {R_N^{(K')} \left( {z,\nu } \right)} \big| \le \frac{{\left| {\cos \left( {\pi \nu } \right)} \right|}}{{\left| {\cos \left( {\pi \Re \left( \nu  \right)} \right)} \right|}}\frac{{\left| {b_N \left( {\Re \left( \nu  \right)} \right)} \right|}}{{\left| z \right|^N }}\mathop {\sup }\limits_{r \geq 1} \big| {\Lambda _{N + \frac{1}{2}} \left( {2zr } \right)} \big|, 
\end{equation}
provided $\left| {\arg z} \right| < \frac{{3\pi }}{2}$, and
\begin{equation}\label{eq221}
\big| {R_N^{(J')} \left( {z,\nu } \right)} \big| \le \frac{{\left| {\cos \left( {\pi \nu } \right)} \right|}}{{\left| {\cos \left( {\pi \Re \left( \nu  \right)} \right)} \right|}}\frac{{\left| {b_N \left( {\Re \left( \nu  \right)} \right)} \right|}}{{\left| z \right|^N }}\mathop {\sup }\limits_{r \geq 1} \big| {\Pi _{N + \frac{1}{2}} \left( {2zr } \right)} \big|,
\end{equation}
provided $\left| {\arg z} \right| < \pi$. If $2\Re \left( \nu  \right)$ is an odd integer, then the limiting values have to be taken in these bounds.
\end{theorem}

We would like to emphasize that the requirement $\left| {\Re \left( \nu  \right)} \right| < N + \frac{1}{2}$ (respectively, $\left| {\Re \left( \nu  \right)} \right| < N - \frac{1}{2}$) in the above theorems is not a serious restriction. Indeed, the index of the numerically least term of the asymptotic expansion \eqref{eq88}, for example, is $n \approx 2\left|z\right|$. Therefore, it is reasonable to choose the optimal $N \approx 2\left|z\right|$, whereas the condition $\nu^2  = o(\left| z \right|)$ has to be fulfilled in order to obtain proper approximations from \eqref{eq70}.

A detailed discussion on the sharpness of our error bounds and their relation to other results in the literature is given in Section \ref{section5}.

In the following, we consider re-expansions for the remainders of the asymptotic expansions of the Bessel and modified Bessel function as well as of their derivatives. Re-expansions for the remainder terms of the asymptotic expansions of the functions $J_0\left(z\right)$, $Y_0\left(z\right)$ and $K_0\left(z\right)$, in order to improve their numerical efficacy, were first derived by Stieltjes \cite{Stieltjes} in 1886 (see also Watson \cite[pp. 213--214]{Watson}). His results were extended to arbitrary real order $\nu$ by Koshliakov \cite{Koshliakov} and Burnett \cite{Burnett} in 1929 and in 1930, respectively. More general expansions were established later, using formal methods, by Airey \cite{Airey} in 1937 and by Dingle \cite{Dingle3} \cite[pp. 441 and 450]{Dingle2} in 1959 (see also \cite{Dempsey}). Their work was placed on rigorous mathematical foundations by Boyd \cite{Boyd} and Olver \cite{Olver} in 1990, who derived an exponentially improved asymptotic expansion for the modified Bessel function $K_\nu \left(z\right)$ valid when $\left|\arg z\right| \leq \pi$. Here, we shall reconsider the result of Boyd and Olver and obtain explicit bounds for the error term of their expansion. Analogous results for the remainders $R_N^{(K')} \left( {z,\nu} \right)$, $R_N^{\left( J \right)} \left( {z,\nu} \right)$ and $R_N^{(J')} \left( {z,\nu} \right)$ are also provided.

For any non-negative integers $N$ and $M$, $M<N$, we write
\begin{equation}\label{eq437}
R_N^{\left( K \right)} \left( {z,\nu } \right) = \left( { - 1} \right)^N \frac{{\cos \left( {\pi \nu } \right)}}{{2^N \pi }}\frac{1}{{z^N }}\sum\limits_{m = 0}^{M - 1} {2^m a_m \left( \nu  \right)\Gamma \left( {N - m} \right)\Lambda _{N - m} \left( {2z} \right)}  + R_{N,M}^{\left( K \right)} \left( {z,\nu } \right),
\end{equation}
\begin{equation}
R_N^{(K')} \left( {z,\nu } \right) = \left( { - 1} \right)^{N + 1} \frac{{\cos \left( {\pi \nu } \right)}}{{2^N \pi }}\frac{1}{{z^N }}\sum\limits_{m = 0}^{M - 1} {2^m b_m \left( \nu  \right)\Gamma \left( {N - m} \right)\Lambda _{N - m} \left( {2z} \right)}  + R_{N,M}^{(K')} \left( {z,\nu } \right),
\end{equation}
\begin{equation}\label{eq737}
R_N^{\left( J \right)} \left( {z,\nu } \right) = \left( { - 1} \right)^{\left\lfloor N/2\right\rfloor} \frac{{\cos \left( {\pi \nu } \right)}}{{2^N \pi }}\frac{1}{{z^N }}\sum\limits_{m = 0}^{M - 1} {2^m a_m \left( \nu  \right)\Gamma \left( {N - m} \right)\Pi _{N - m} \left( {2z} \right)}  + R_{N,M}^{\left( J \right)} \left( {z,\nu } \right)
\end{equation}
and
\begin{equation}\label{eq440}
R_N^{(J')} \left( {z,\nu } \right) =  \left( { - 1} \right)^{\left\lfloor N/2\right\rfloor+1} \frac{{\cos \left( {\pi \nu } \right)}}{{2^N \pi }}\frac{1}{{z^N }}\sum\limits_{m = 0}^{M - 1} {2^m b_m \left( \nu  \right)\Gamma \left( {N - m} \right)\Pi _{N - m} \left( {2z} \right)}  + R_{N,M}^{(J')} \left( {z,\nu } \right).
\end{equation}
The re-expansion \eqref{eq437} is equivalent to that studied by Boyd and Olver. Koshliakov's, Burnett's and Airey's expansions may be deduced from \eqref{eq437} by choosing $N\approx 2\left|z\right|$ and employing the known asymptotics of the gamma functions $\Gamma \left( {N - m} \right)$ and the basic terminants $\Lambda _{N - m} \left( {2z} \right)$ (see, e.g., \cite[eq. (6.4.58), p. 75]{Temme} and \cite[eq. (2.14)]{Olver2}).

In the following two theorems, we give bounds for the remainders $R_{N,M}^{\left( K \right)} \left( {z,\nu } \right)$, $R_{N,M}^{\left( J \right)} \left( {z,\nu } \right)$, $R_{N,M}^{(K')} \left( {z,\nu } \right)$ and $R_{N,M}^{(J')} \left( {z,\nu } \right)$.

\begin{theorem}\label{theorem10}
Let $N$ and $M$ be arbitrary non-negative integers such that $M < N$, and let $\nu$ be a complex number satisfying $\left| {\Re \left( \nu  \right)} \right| < M + \frac{1}{2}$. Then we have
\begin{gather}\label{eq08}
\begin{split}
\big| {R_{N,M}^{\left( K \right)} \left( {z,\nu } \right)} \big| \le \; & \frac{{\left| {\cos \left( {\pi \nu } \right)} \right|}}{{2^N \pi }}\frac{1}{{\left| z \right|^N }}2^M \left| z \right|^M \big| {R_M^{\left( K \right)} \left( {\left|z\right|,\nu } \right)} \big| \Gamma \left( {N - M} \right) \left| {\Lambda _{N - M} \left( {2z} \right)} \right| \\ & + \frac{{\left| {\cos \left( {\pi \nu } \right)} \right|}}{{2^N \pi }}\frac{1}{{\left| z \right|^N }}2^M \frac{{\left| {\cos \left( {\pi \nu } \right)} \right|}}{{\left| {\cos \left( {\pi \Re \left( \nu  \right)} \right)} \right|}}\left| {a_M \left( {\Re \left( \nu  \right)} \right)} \right|\Gamma \left( {N - M} \right),
\end{split}
\end{gather}
for $\left|\arg z\right| \leq \pi$, and
\begin{gather}\label{eq822}
\begin{split}
\big| {R_{N,M}^{\left( J \right)} \left( {z,\nu } \right)} \big| \le \; & \frac{{\left| {\cos \left( {\pi \nu } \right)} \right|}}{{2^N \pi }}\frac{1}{{\left| z \right|^N }}2^M \left| z \right|^M \big| {R_M^{\left( K \right)} \left( {\left| z \right|,\nu } \right)} \big|\Gamma \left( {N - M} \right)\left| {\Pi _{N - M} \left( {2z} \right)} \right| \\ & + \frac{{\left| {\cos \left( {\pi \nu } \right)} \right|}}{{2^N \pi }}\frac{1}{{\left| z \right|^N }}2^M \frac{{\left| {\cos \left( {\pi \nu } \right)} \right|}}{{\left| {\cos \left( {\pi \Re \left( \nu  \right)} \right)} \right|}}\left| {a_M \left( {\Re \left( \nu  \right)} \right)} \right|\Gamma \left( {N - M} \right),
\end{split}
\end{gather}
for $\left|\arg z\right| \leq \frac{\pi}{2}$. If $2\Re \left( \nu  \right)$ is an odd integer, then the limiting values are taken in these bounds.
\end{theorem}

\begin{theorem}\label{theorem11}
Let $N$ and $M$ be arbitrary positive integers such that $M < N$, and let $\nu$ be a complex number satisfying $\left| {\Re \left( \nu  \right)} \right| < M - \frac{1}{2}$. Then we have
\begin{gather}\label{eq841}
\begin{split}
\big| {R_{N,M}^{(K')} \left( {z,\nu } \right)} \big| \le \; & \frac{{\left| {\cos \left( {\pi \nu } \right)} \right|}}{{2^N \pi }}\frac{1}{{\left| z \right|^N }}2^M \left| z \right|^M \big| {R_M^{(K')} \left( {\left|z\right|,\nu } \right)} \big| \Gamma \left( {N - M} \right) \left| {\Lambda _{N - M} \left( {2z} \right)} \right| \\ & + \frac{{\left| {\cos \left( {\pi \nu } \right)} \right|}}{{2^N \pi }}\frac{1}{{\left| z \right|^N }}2^M \frac{{\left| {\cos \left( {\pi \nu } \right)} \right|}}{{\left| {\cos \left( {\pi \Re \left( \nu  \right)} \right)} \right|}}\left| {b_M \left( {\Re \left( \nu  \right)} \right)} \right|\Gamma \left( {N - M} \right),
\end{split}
\end{gather}
for $\left|\arg z\right| \leq \pi$, and
\begin{gather}\label{eq842}
\begin{split}
\big| {R_{N,M}^{(J')} \left( {z,\nu } \right)} \big| \le \; & \frac{{\left| {\cos \left( {\pi \nu } \right)} \right|}}{{2^N \pi }}\frac{1}{{\left| z \right|^N }}2^M \left| z \right|^M \big| {R_M^{(K')} \left( {\left| z \right|,\nu } \right)} \big|\Gamma \left( {N - M} \right)\left| {\Pi _{N - M} \left( {2z} \right)} \right| \\ & + \frac{{\left| {\cos \left( {\pi \nu } \right)} \right|}}{{2^N \pi }}\frac{1}{{\left| z \right|^N }}2^M \frac{{\left| {\cos \left( {\pi \nu } \right)} \right|}}{{\left| {\cos \left( {\pi \Re \left( \nu  \right)} \right)} \right|}}\left| {b_M \left( {\Re \left( \nu  \right)} \right)} \right|\Gamma \left( {N - M} \right),
\end{split}
\end{gather}
for $\left|\arg z\right| \leq \frac{\pi}{2}$. If $2\Re \left( \nu  \right)$ is an odd integer, then the limiting values are taken in these bounds.
\end{theorem}

We remark that Boyd also considered the problem of bounding the remainder $R_{N,M}^{\left( K \right)} \left( {z,\nu } \right)$ in the special case when $\nu$ is real and $\left|\nu\right|<\frac{1}{2}$. His result can be recovered from \eqref{eq08} by using the estimate $\left| z \right|^M \big| {R_M^{\left( K \right)} \left( {\left| z \right|,\nu } \right)} \big| \le \left| {a_M \left( \nu  \right)} \right|$ (cf. equation \eqref{eq711}). A brief discussion on the sharpness of the error bound \eqref{eq08} is given in Section \ref{section5}.

In the case when $z$ is positive and $\nu$ is real, we shall show that the remainder terms $R_{N,M}^{\left( K \right)} \left( {z,\nu } \right)$ and $R_{N,M}^{\left( J \right)} \left( {z,\nu } \right)$ (respectively $R_{N,M}^{(K')} \left( {z,\nu } \right)$ and $R_{N,M}^{(J')} \left( {z,\nu } \right)$) do not exceed the corresponding first neglected terms in absolute value and have the same sign provided that $\left| {\nu} \right| < M + \frac{1}{2}$ (respectively $\left| {\nu} \right| < M - \frac{1}{2}$). More precisely, we will prove that the following theorems hold.

\begin{theorem}\label{theorem12} Let $N$ and $M$ be arbitrary non-negative integers satisfying $M < N$. Let $z$ be a positive real number and $\nu$ be an arbitrary real number such that $\left| {\nu} \right| < M + \frac{1}{2}$. Then
\begin{equation}\label{eq575}
R_{N,M}^{\left( K \right)} \left( {z,\nu } \right) = \left( { - 1} \right)^N \frac{{\cos \left( {\pi \nu } \right)}}{{2^N \pi }}\frac{1}{{z^N }}2^M a_M \left( \nu  \right)\Gamma \left( {N - M} \right)\Lambda _{N - M} \left( {2z} \right)\Theta _{N,M}^{\left( K \right)} \left( {z,\nu } \right)
\end{equation}
and
\begin{equation}\label{eq576}
R_{N,M}^{\left( J \right)} \left( {z,\nu } \right) = \left( { - 1} \right)^{\left\lfloor N/2\right\rfloor} \frac{{\cos \left( {\pi \nu } \right)}}{{2^N \pi }}\frac{1}{{z^N }}2^M a_M \left( \nu  \right)\Gamma \left( {N - M} \right)\Pi _{N - M} \left( {2z} \right)\Theta _{N,M}^{\left( J \right)} \left( {z,\nu } \right),
\end{equation}
where $0 < \Theta _{N,M}^{\left( K \right)} \left( {z,\nu } \right),\Theta _{N,M}^{\left( J \right)} \left( {z,\nu } \right) < 1$ are appropriate numbers that depend on $z, \nu$, $N$ and $M$.
\end{theorem}

\begin{theorem}\label{theorem13} Let $N$ and $M$ be arbitrary positive integers satisfying $M < N$. Let $z$ be a positive real number and $\nu$ be an arbitrary real number such that $\left| {\nu} \right| < M - \frac{1}{2}$. Then
\begin{equation}
R_{N,M}^{(K')} \left( {z,\nu } \right) = \left( { - 1} \right)^{N + 1} \frac{{\cos \left( {\pi \nu } \right)}}{{2^N \pi }}\frac{1}{{z^N }}2^M b_M \left( \nu  \right)\Gamma \left( {N - M} \right)\Lambda _{N - M} \left( {2z} \right)\Theta _{N,M}^{(K')} \left( {z,\nu } \right)
\end{equation}
and
\begin{equation}\label{eq577}
R_{N,M}^{(J')} \left( {z,\nu } \right) =  \left( { - 1} \right)^{\left\lfloor N/2\right\rfloor+1} \frac{{\cos \left( {\pi \nu } \right)}}{{2^N \pi }}\frac{1}{{z^N }}2^M b_M \left( \nu  \right)\Gamma \left( {N - M} \right)\Pi _{N - M} \left( {2z} \right)\Theta _{N,M}^{(J')} \left( {z,\nu } \right),
\end{equation}
where $0 < \Theta _{N,M}^{(K')} \left( {z,\nu } \right) , \Theta _{N,M}^{(J')} \left( {z,\nu } \right) < 1$ are appropriate numbers that depend on $z, \nu$, $N$ and $M$.
\end{theorem}

The best accuracy available from the re-expansions \eqref{eq437}--\eqref{eq440} can be obtained by setting $N = 4\left| z \right| + \rho$ and $M = 2\left| z \right| + \sigma$, with $\rho$ and $\sigma$ being appropriate (fixed) real numbers. With these choices of $N$ and $M$, Theorems \ref{theorem10} and \ref{theorem11}, combined with Stirling's approximation for the gamma function and the bounds \eqref{eq368} and \eqref{eq369}, imply that
\begin{equation}\label{eq443}
R_{N,M}^{\left( K \right)} \left( {z,\nu } \right),R_{N,M}^{(K')} \left( {z,\nu } \right)=\mathcal{O}_{\nu ,\rho ,\sigma } \big(\left| z \right|^{ - \frac{1}{2}} e^{ - 4\left| z \right|} \big)
\end{equation}
as $z\to \infty$ in the sector $\left| {\arg z} \right| \le \pi$, and
\begin{equation}\label{eq444}
R_{N,M}^{\left( J \right)} \left( {z,\nu } \right),R_{N,M}^{(J')} \left( {z,\nu } \right)=\mathcal{O}_{\nu ,\rho ,\sigma } \big(\left| z \right|^{ - \frac{1}{2}} e^{ - 4\left| z \right|} \big)
\end{equation}
as $z\to \infty$ in the sector $\left| {\arg z} \right| \le \frac{\pi}{2}$. (Throughout this paper, we use subscripts in the $\mathcal{O}$ notations to indicate the dependence of the implied constant on certain parameters.) The estimates \eqref{eq443} and \eqref{eq444} can alternatively be deduced from the hyperasymptotic theory of second order linear differential equations, discussed in the papers \cite{Hyper1} and \cite{Hyper2}.

Results analogous to those given in this paper for other ranges of $\arg z$ can be obtained by making use of the analytic continuation formulae for the Hankel and Bessel functions (see, e.g., \cite[\S 10.11 and \S 10.34]{NIST}). Also, by setting $\nu = \frac{1}{3}$ or $\nu = \frac{2}{3}$, one can derive error bounds for the asymptotic expansions of the Airy functions and their derivatives, respectively, but we shall not pursue the details here.

As a closing remark, we would like to point out that the techniques used in this paper should generalize to integrals over steepest descent contours. Indeed, the theory developed by Berry and Howls (\cite{Berry}, especially equations (12) and (15)) and Howls (\cite{Howls}, especially equations (17) and (35)) enables us to derive representations analogous to \eqref{eq9} for the remainder terms of asymptotic expansions arising from an application of the method of steepest descents. The estimation of these new representations would then provide error bounds for the method of steepest descents which are presumably more general and sharper than those that exist in the literature (cf., in particular, \cite{Boyd2}).

The remaining part of the paper is structured as follows. In Section \ref{section2}, we prove the representations for the remainder terms stated in Theorems \ref{theorem1}--\ref{theorem9}. In Section \ref{section3}, we prove the error bounds given in Theorems \ref{theorem3}--\ref{theorem7}. Section \ref{section4} discusses the proof of the bounds for the remainders of the re-expansions stated in Theorems \ref{theorem10}--\ref{theorem13}. The paper concludes with a discussion in Section \ref{section5}.

\section{Proof of the representations for the remainder terms}\label{section2}

In this section, we prove the representations for the remainder terms stated in Theorems \ref{theorem1}--\ref{theorem9}. In order to prove Theorems \ref{theorem1} and \ref{theorem2}, we require certain estimates for the regularized hypergeometric function given in the following lemma.

\begin{lemma}\label{lemma1} Let $\nu$ and $\lambda$ be arbitrary fixed complex numbers such that $\Re \left( \lambda  \right) > 0$ or $\lambda=0$. Then
\begin{equation}\label{eq20}
{\bf F}\left( {\nu  + \frac{1}{2}, - \nu  + \frac{1}{2};\lambda ; - t} \right) =  \begin{cases} \mathcal{O}_{\nu ,\lambda } \left( 1 \right) & \text{if } \Re \left( \lambda  \right)> 0, \\  \mathcal{O}_{\nu } \left( t \right) & \text{if } \lambda= 0\end{cases}
\end{equation}
as $t\to 0+$, and
\begin{equation}\label{eq21}
{\bf F}\left( {\nu  + \frac{1}{2}, - \nu  + \frac{1}{2};\lambda ; - t} \right) = \mathcal{O}_{\nu ,\lambda } \big( {\left( {1 + t} \right)^{\left| {\Re \left( \nu  \right)} \right| - \frac{1}{2}} \log \left( {1 + t} \right)} \big)
\end{equation}
as $t\to +\infty$.
\end{lemma}

\begin{proof} The estimate \eqref{eq20} is a consequence of the power series expansion of the regularized hypergeometric function. To prove the estimate \eqref{eq21}, we use the known linear transformation formula for the regularized hypergeometric function (see, e.g., \cite[15.8.E1]{NIST}) and obtain
\begin{equation}\label{eq22}
{\bf F}\left( {\nu  + \frac{1}{2}, - \nu  + \frac{1}{2};\lambda ; - t} \right)  = \left( {1 + t} \right)^{ - \nu  - \frac{1}{2}}{\bf F}\left( {\nu  + \frac{1}{2},\lambda  + \nu  - \frac{1}{2};\lambda ;\frac{t}{{1 + t}}} \right)
\end{equation}
for $t>0$. Now, by employing the known behaviour of the regularized hypergeometric function near one (see, e.g., \cite[\S 15.4(ii)]{NIST}), we find that
\[
{\bf F}\left( {\nu  + \frac{1}{2},\lambda  + \nu  - \frac{1}{2};\lambda ;\frac{t}{{1 + t}}} \right) = \begin{cases} \mathcal{O}_{\nu ,\lambda } \left( 1 \right) & \text{if } \Re \left( \nu  \right) < 0 \text{ or } \Re \left( \nu  \right) = 0,\nu  \ne 0, \\ \mathcal{O}_\lambda  \left( {\log \left( {1 + t} \right)} \right) & \text{if } \nu = 0, \\ \mathcal{O}_{\nu ,\lambda } \big( {\left( {1 + t} \right)^{2\Re \left( \nu  \right)} } \big) & \text{if } \Re \left( \nu  \right) > 0\end{cases}
\]
as $t\to +\infty$. Combining these estimates with \eqref{eq22}, the result claimed in \eqref{eq21} follows.
\end{proof}

We continue with the proof of \eqref{eq9}. The modified Bessel function can be represented in terms of the regularized hypergeometric function as
\[
K_\nu  \left( u \right) = \pi^{\frac{1}{2}} e^{ - u} \left(2u\right)^{\lambda  - \frac{1}{2}} \int_0^{ + \infty } {e^{ - 2ut} t^{\lambda  - 1} {\bf F} \left( {\nu  + \frac{1}{2}, - \nu  + \frac{1}{2};\lambda ; - t} \right)dt},
\]
for $|\arg u|<\frac{\pi}{2}$ and with an arbitrary complex $\lambda$, $\Re \left( \lambda  \right) > 0$ \cite[ent. 19.2, p. 195]{Oberhettinger}. Inserting this representation into \eqref{eq1} and changing the order of integration, we deduce
\[
R_N^{\left( K \right)} \left( {z,\nu } \right) = \left( { - 1} \right)^N \frac{{\cos \left( {\pi \nu } \right)}}{\pi }\frac{2^\lambda}{{z^N }}\int_0^{ + \infty } {t^{\lambda  - 1} {\bf F} \left( {\nu  + \frac{1}{2}, - \nu  + \frac{1}{2};\lambda ; - t} \right)\int_0^{ + \infty } {\frac{{u^{N + \lambda  - 1} e^{ - 2\left( {1 + t} \right)u} }}{{1 + u/z}}du} dt} .
\]
The change in the order of integration is justified because the infinite double integrals are absolutely convergent, which can be seen by appealing to Lemma \ref{lemma1}. The inner integral is expressible in terms of the basic terminant $\Lambda _p\left(w\right)$ since
\[
\int_0^{ + \infty } {\frac{{u^{N + \lambda  - 1} e^{ - 2\left( {1 + t} \right)u} }}{{1 + u/z}}du}  = \frac{{\Gamma \left( {N + \lambda } \right)}}{{2^{N + \lambda } \left( {1 + t} \right)^{N + \lambda } }}\Lambda _{N + \lambda} \left( {2z\left( {1 + t} \right)} \right)
\]
(cf. equation \eqref{eq355}), and therefore we obtain
\begin{align*}
R_N^{\left( K \right)} \left( {z,\nu } \right) = & \left( { - 1} \right)^N \frac{{\cos \left( {\pi \nu } \right)}}{\pi }\frac{{\Gamma \left( {N + \lambda } \right)}}{{2^N}}\frac{1}{{z^N }} \\ & \times \int_0^{ + \infty } {\frac{{t^{\lambda  - 1} }}{{\left( {1 + t} \right)^{N + \lambda } }}{\bf F} \left( {\nu  + \frac{1}{2}, - \nu  + \frac{1}{2};\lambda ; - t} \right)\Lambda _{N + \lambda} \left( {2z\left( {1 + t} \right)} \right)dt} .
\end{align*}
Since $\Lambda _p \left( w \right) = \mathcal{O}\left( 1 \right)$ as $w \to \infty$ in the sector $\left| {\arg w} \right| < \frac{{3\pi }}{2}$, by analytic continuation, this representation is valid in a wider range than \eqref{eq1}, namely in $\left| {\arg z} \right| < \frac{{3\pi }}{2}$.

Let us now turn our attention to the proof of \eqref{eq45}. It is enough to show that
\begin{align*}
& \mathop {\lim }\limits_{\lambda  \to 0 + } \int_0^{ + \infty } {\frac{{t^{\lambda  - 1} }}{{\left( {1 + t} \right)^{N + \lambda } }}{\bf F}\left( {\nu  + \frac{1}{2}, - \nu  + \frac{1}{2};\lambda ; - t} \right)\Lambda _{N + \lambda} \left( {2z\left( {1 + t} \right)} \right)dt} \\ & = \Lambda _{N } \left( {2z} \right) + \int_0^{ + \infty } {\frac{{t^{ - 1} }}{{\left( {1 + t} \right)^N }}{\bf F}\left( {\nu  + \frac{1}{2}, - \nu  + \frac{1}{2};0; - t} \right)\Lambda _{N} \left( {2z\left( {1 + t} \right)} \right)dt} ,
\end{align*}
for any positive integer $N$ and fixed complex numbers $z$ and $\nu$ satisfying $\left| {\arg z} \right| < \frac{{3\pi }}{2}$ and $\left| {\Re \left( \nu  \right)} \right| < N + \frac{1}{2}$. Note that the convergence of the integral in the second line is guaranteed by Lemma \ref{lemma1} and the boundedness of the basic terminant. By a simple algebraic manipulation and an application of the beta integral (cf. \cite[5.12.E3]{NIST}), we can assert that
\begin{align*}
& \int_0^{ + \infty } {\frac{{t^{\lambda  - 1} }}{{\left( {1 + t} \right)^{N + \lambda } }}{\bf F}\left( {\nu  + \frac{1}{2}, - \nu  + \frac{1}{2};\lambda ; - t} \right)\Lambda _{N + \lambda } \left( {2z\left( {1 + t} \right)} \right)dt} \\ = \; & \int_0^{ + \infty } {\frac{{t^{\lambda  - 1} }}{{\left( {1 + t} \right)^{N + \lambda } }}\frac{1}{{\Gamma \left( \lambda  \right)}}\Lambda _{N + \lambda } \left( {2z} \right)dt} 
\\ & + \int_0^{ + \infty } {\frac{{t^{\lambda  - 1} }}{{\left( {1 + t} \right)^{N + \lambda } }}\left( {{\bf F}\left( {\nu  + \frac{1}{2}, - \nu  + \frac{1}{2};\lambda ; - t} \right)\Lambda _{N + \lambda } \left( {2z\left( {1 + t} \right)} \right) - \frac{1}{{\Gamma \left( \lambda  \right)}}\Lambda _{N + \lambda } \left( {2z} \right)} \right)dt} 
\\  =\; & \frac{{\Gamma \left( N \right)}}{{\Gamma \left( {N + \lambda } \right)}}\Lambda _{N + \lambda } \left( {2z} \right) \\ & + \int_0^{ + \infty } {\frac{{t^{\lambda  - 1} }}{{\left( {1 + t} \right)^{N + \lambda } }}\left( {{\bf F}\left( {\nu  + \frac{1}{2}, - \nu  + \frac{1}{2};\lambda ; - t} \right)\Lambda _{N + \lambda } \left( {2z\left( {1 + t} \right)} \right) - \frac{1}{{\Gamma \left( \lambda  \right)}}\Lambda _{N + \lambda } \left( {2z} \right)} \right)dt} .
\end{align*}
By continuity, we have
\[
\mathop {\lim }\limits_{\lambda  \to 0 + } \frac{{\Gamma \left( N \right)}}{{\Gamma \left( {N + \lambda } \right)}}\Lambda _{N + \lambda } \left( {2z} \right) = \Lambda _{N } \left( {2z} \right),
\]
and thus, it remains to prove that
\begin{gather}\label{eq111}
\begin{split}
& \mathop {\lim }\limits_{\lambda  \to 0 + } \int_0^{ + \infty } {\frac{{t^{\lambda  - 1} }}{{\left( {1 + t} \right)^{N + \lambda } }}\left( {{\bf F}\left( {\nu  + \frac{1}{2}, - \nu  + \frac{1}{2};\lambda ; - t} \right)\Lambda _{N + \lambda } \left( {2z\left( {1 + t} \right)} \right) - \frac{1}{{\Gamma \left( \lambda  \right)}}\Lambda _{N + \lambda } \left( {2z} \right)} \right)dt} \\ & = \int_0^{ + \infty } {\frac{{t^{ - 1} }}{{\left( {1 + t} \right)^N }}{\bf F}\left( {\nu  + \frac{1}{2}, - \nu  + \frac{1}{2};0; - t} \right)\Lambda _{N } \left( {2z\left( {1 + t} \right)} \right)dt}.
\end{split}
\end{gather}
We show that the absolute value of the integrand in the first line can be bounded pointwise by an absolutely integrable function uniformly with respect to bounded positive values of $\lambda$. Consequently, by the Lebesgue dominated convergence theorem \cite[Theorem 1.4.49, p. 111]{Tao}, the order of the limit and integration can be interchanged and the required result follows. Assume therefore that $\lambda$ is bounded, say $0<\lambda \ll 1$. Then, by Taylor's theorem, we have
\[
{\bf F}\left( {\nu  + \frac{1}{2}, - \nu  + \frac{1}{2};\lambda ; - t} \right)\Lambda _{N + \lambda } \left( {2z\left( {1 + t} \right)} \right) = \frac{1}{{\Gamma \left( \lambda  \right)}}\Lambda _{N + \lambda } \left( {2z} \right) + \mathcal{O}_{z,\nu,N} \left( t \right)
\]
as $t\to 0+$. Consequently, the integrand in the first line of \eqref{eq111} is $\mathcal{O}_{z,\nu,N} \left( {t^\lambda  } \right) = \mathcal{O}_{z,\nu,N} \left( 1 \right)$ as $t\to 0+$ and, appealing to Lemma \ref{lemma1} and to the boundedness of the basic terminant, it is $\mathcal{O}_{z,\nu,N} ( \left( {1 + t} \right)^{\left| {\Re \left( \nu  \right)} \right| - N - \frac{3}{2}} \log \left( {1 + t} \right) ) + \mathcal{O}_{z,\nu,N} ( \left( {1 + t} \right)^{ - N - 1} )$ as $t\to+\infty$. Therefore, there exists a positive constant $C_{z,\nu,N}$, independent of $t$ and $\lambda$ such that the absolute value of the integrand in the first line of \eqref{eq111} is bounded from above by
\begin{equation}\label{eq720}
C_{z,\nu,N} \big(\left( {1 + t} \right)^{\left| {\Re \left( \nu  \right)} \right| - N - \frac{3}{2}} \log \left( {1 + t} \right) + \left( {1 + t} \right)^{ - N - 1} \big)
\end{equation}
for any $t>0$. Since $N\geq 1$ and $\left| {\Re \left( \nu  \right)} \right| < N + \frac{1}{2}$, the function \eqref{eq720} is absolutely integrable on the positive $t$-axis and the proof is complete.

To prove the representations \eqref{eq54} and \eqref{eq55}, we can proceed as follows. The Bessel function $J_\nu  \left( z \right)$ is related to the modified Bessel function $K_\nu  \left( z \right)$ via the connection formula
\begin{gather}\label{eq080}
\begin{split}
\pi iJ_\nu  \left( z \right) = \; & e^{ - \frac{\pi }{2}i\nu } K_\nu  \left( {ze^{ - \frac{\pi }{2}i} } \right) - e^{\frac{\pi }{2}i\nu } K_\nu  \left( {ze^{\frac{\pi }{2}i} } \right)
\\ =\; & i\cos \omega \left( {e^{\frac{\pi }{4}i} e^{iz} K_\nu  \left( {ze^{\frac{\pi }{2}i} } \right) + e^{ - \frac{\pi }{4}i} e^{ - iz} K_\nu  \left( {ze^{ - \frac{\pi }{2}i} } \right)} \right) \\ &+ \sin \omega \left( {e^{\frac{\pi }{4}i} e^{iz} K_\nu  \left( {ze^{\frac{\pi }{2}i} } \right) - e^{ - \frac{\pi }{4}i} e^{ - iz} K_\nu  \left( {ze^{ - \frac{\pi }{2}i} } \right)} \right),
\end{split}
\end{gather}
with $\left| {\arg z} \right| \leq \frac{\pi }{2}$ and $\omega  = z - \frac{\pi }{2}\nu  - \frac{\pi }{4}$ \cite[10.27.E9]{NIST}. Now let $N$ and $M$ be arbitrary non-negative integers. We express the functions $K_\nu  \left( {ze^{ \pm \frac{\pi }{2}i} } \right)$ in the first parenthesis as
\[
K_\nu  \left( {ze^{ \pm \frac{\pi }{2}i} } \right) = \left( {\frac{\pi }{{2ze^{ \pm \frac{\pi }{2}i} }}} \right)^{\frac{1}{2}} e^{ \mp iz} \left( {\sum\limits_{n = 0}^{2N - 1} {\frac{{a_n \left( \nu  \right)}}{{\left( { \pm iz} \right)^n }}}  + R_{2N}^{\left( K \right)} \left( {ze^{ \pm \frac{\pi }{2}i} ,\nu } \right)} \right)
\]
and in the second parenthesis as
\[
K_\nu  \left( {ze^{ \pm \frac{\pi }{2}i} } \right) = \left( {\frac{\pi }{{2ze^{ \pm \frac{\pi }{2}i} }}} \right)^{\frac{1}{2}} e^{ \mp iz} \left( {\sum\limits_{m = 0}^{2M} {\frac{{a_m \left( \nu  \right)}}{{\left( { \pm iz} \right)^m }}}  + R_{2M+1}^{\left( K \right)} \left( {ze^{ \pm \frac{\pi }{2}i} ,\nu } \right)} \right),
\]
(cf. equation \eqref{eq70}) which, after some cancellations, yields
\begin{gather}\label{eq081}
\begin{split}
J_\nu  \left( z \right) & = \left( {\frac{2}{{\pi z}}} \right)^{\frac{1}{2}} \left( \cos \omega \left( {\sum\limits_{n = 0}^{N - 1} {\left( { - 1} \right)^n \frac{{a_{2n} \left( \nu  \right)}}{{z^{2n} }}}  + R_{2N}^{\left( J \right)} \left( {z,\nu } \right)} \right) \right. \\ & \hspace{130pt} \left.- \sin \omega \left( {\sum\limits_{m = 0}^{M - 1} {\left( { - 1} \right)^m \frac{{a_{2m + 1} \left( \nu  \right)}}{{z^{2m + 1} }}}  - R_{2M + 1}^{\left( J \right)} \left( {z,\nu } \right)} \right) \right).
\end{split}
\end{gather}
The remainder terms $R_{2N}^{\left( J \right)} \left( {z,\nu } \right)$ and $R_{2M + 1}^{\left( J \right)} \left( {z,\nu } \right)$ are directly related to the remainders $R_{2N}^{\left( K \right)} \left( {ze^{ \pm \frac{\pi }{2}i} ,\nu } \right)$ and $R_{2M+1}^{\left( K \right)} \left( {ze^{ \pm \frac{\pi }{2}i} ,\nu } \right)$ since
\begin{equation}\label{eq56}
2R_{2N}^{\left( J \right)} \left( {z,\nu } \right) = R_{2N}^{\left( K \right)} \left( {ze^{\frac{\pi }{2}i} ,\nu } \right) + R_{2N}^{\left( K \right)} \left( {ze^{ - \frac{\pi }{2}i} ,\nu } \right)
\end{equation}
and
\begin{equation}\label{eq57}
2iR_{2M + 1}^{\left( J \right)} \left( {z,\nu } \right) = R_{2M+1}^{\left( K \right)} \left( {ze^{\frac{\pi }{2}i} ,\nu } \right) - R_{2M+1}^{\left( K \right)} \left( {ze^{ - \frac{\pi }{2}i} ,\nu } \right),
\end{equation}
respectively. To prove \eqref{eq54}, we substitute into the right-hand side of \eqref{eq56}, respectively \eqref{eq57}, the representation \eqref{eq9} with $N=2N$, respectively $N=2M+1$. The desired result then follows using analytic continuation in $z$ and the definition of the basic terminant $\Pi_p\left(w\right)$. Formula \eqref{eq55} can be proved analogously, applying the representation \eqref{eq45}.

Formula \eqref{eq91} can be obtained by substituting into the right-hand side of \eqref{eq56}, respectively \eqref{eq57}, the representation \eqref{eq1} with $N=2N$, respectively $N=2M+1$. The representation \eqref{eq92} follows by combining formula \eqref{eq1} with the functional relations \eqref{eq101} and $- 2K'_\nu  \left( t \right) = K_{\nu + 1} \left( t \right) + K_{\nu  - 1} \left( t \right)$. Similarly, formula \eqref{eq93} can be obtained by combining \eqref{eq91} with the functional relations \eqref{eq102} and $- 2K'_\nu  \left( t \right) = K_{\nu + 1} \left( t \right) + K_{\nu  - 1} \left( t \right)$.

\section{Proof of the error bounds}\label{section3}

In this section, we prove the bounds for the remainder terms $R_N^{\left( K \right)} \left( {z,\nu } \right)$, $R_N^{\left( J \right)} \left( {z,\nu } \right)$, $R_N^{(K')} \left( {z,\nu } \right)$ and $R_N^{(J')} \left( {z,\nu } \right)$ given in Theorems \ref{theorem3}--\ref{theorem7}. To this end, we shall state and prove a series of lemmata.

\begin{lemma}\label{lemma2} If $t$ is positive and $\nu$ and $\lambda$ are real numbers such that $\lambda  > \max \left( {0,\frac{1}{2} - \left| \nu  \right|} \right)$, then
\[
{\bf F}\left( {\nu  + \frac{1}{2}, - \nu  + \frac{1}{2};\lambda ; - t} \right) \geq 0.
\]
\end{lemma}

\begin{proof} Since ${\bf F}\left( {\nu  + \frac{1}{2}, - \nu  + \frac{1}{2};\lambda; - t} \right) = {\bf F}\left( { - \nu  + \frac{1}{2},\nu  + \frac{1}{2};\lambda; - t} \right)$, we may assume that $\nu$ is non-negative. Using \eqref{eq22} and the power series expansion of the regularized hypergeometric function, we deduce that
\begin{equation}\label{eq44}
{\bf F}\left( {\nu  + \frac{1}{2}, - \nu  + \frac{1}{2};\lambda ; - t} \right)  = \frac{{\left( {1 + t} \right)^{ - \nu  - \frac{1}{2}} }}{{\Gamma \left( {\nu  + \frac{1}{2}} \right)\Gamma \left( {\lambda  + \nu  - \frac{1}{2}} \right)}}\sum\limits_{n = 0}^\infty  {\frac{{\Gamma \left( {\nu  + \frac{1}{2} + n} \right)\Gamma \left( {\lambda  + \nu  - \frac{1}{2} + n} \right)}}{{\Gamma \left( {\lambda  + n} \right)\Gamma \left( {n + 1} \right)}}\left( {\frac{t}{{1 + t}}} \right)^n } ,
\end{equation}
for any $t>0$. Since the right-hand side is non-negative if $\lambda  + \nu  - \frac{1}{2}$ is positive, i.e., if $\lambda  > \frac{1}{2} - \nu$, the statement of the lemma follows.
\end{proof}

We remark that, by direct numerical computation, it can be verified that when $\left| \nu  \right| < \frac{1}{2}$ and $0<\lambda  < \frac{1}{2} - \left| \nu  \right|$, the quantity ${\bf F}\left( {\nu  + \frac{1}{2}, - \nu  + \frac{1}{2};\lambda ; - t} \right)$, as a function of $t>0$, may take both positive and negative values. Therefore, in general, the condition $\lambda  > \frac{1}{2} - \left| \nu  \right|$ in Lemma \ref{lemma2} cannot be weakened.

\begin{lemma}\label{lemma3} If $t$ is positive and $\nu$ is real, then
\begin{equation}\label{eq46}
{\bf F}\left( {\nu  + \frac{3}{2}, - \nu  - \frac{1}{2};0; - t} \right) + {\bf F}\left( {\nu  - \frac{1}{2}, - \nu  + \frac{3}{2};0; - t} \right) \geq 0.
\end{equation}
\end{lemma}

\begin{proof} First, we show that
\begin{equation}\label{eq48}
{\bf F}\left( {\nu  + \frac{1}{2}, - \nu  + \frac{1}{2};0; - t} \right) \begin{cases} <0 & \text{if } 0\leq \left|\nu\right| <\frac{1}{2}, \\ \geq 0 & \text{if }  \left|\nu\right|\geq \frac{1}{2} \end{cases}
\end{equation}
for any $t>0$. Since ${\bf F}\left( {\nu  + \frac{1}{2}, - \nu  + \frac{1}{2};0; - t} \right) = {\bf F}\left( { - \nu  + \frac{1}{2},\nu  + \frac{1}{2};0; - t} \right)$, we may assume that $\nu$ is non-negative. In the case that $\lambda=0$, the expansion \eqref{eq44} may be re-written in the form
\[
{\bf F}\left( {\nu  + \frac{1}{2}, - \nu  + \frac{1}{2};0; - t} \right) = \left( {\nu  - \frac{1}{2}} \right)\frac{{\left( {1 + t} \right)^{ - \nu  - \frac{1}{2}} }}{{\Gamma ^2 \left( {\nu  + \frac{1}{2}} \right)}}\sum\limits_{n = 0}^\infty  {\frac{{\Gamma \left( {\nu  + \frac{3}{2} + n} \right)\Gamma \left( {\nu  + \frac{1}{2} + n} \right)}}{{\Gamma \left( {n + 1} \right)\Gamma \left( {n + 2} \right)}}\left( {\frac{t}{{1 + t}}} \right)^{n + 1} } .
\]
The sign of the expression on the right-hand side is determined by the factor $\nu-\frac{1}{2}$, whence \eqref{eq48} indeed holds.

Consider now the proof of \eqref{eq46}. We may assume, without loss of generality, that $\nu$ is non-negative since the left-hand side of \eqref{eq46} does not change if $\nu$ is replaced by $-\nu$. By \eqref{eq48}, the terms on the left-hand side of \eqref{eq46} are non-negative for positive $t$, except, perhaps, when $\frac{1}{2} < \nu < \frac{3}{2}$. We treat this case separately. The regularized hypergeometric function ${\bf F}\left( {\nu  + \frac{1}{2}, - \nu + \frac{1}{2};0 ; - t} \right)$ can be written in terms of the associated Legendre function as
\[
{\bf F}\left( {\nu  + \frac{1}{2}, - \nu + \frac{1}{2};0 ; - t} \right)=\left( {\frac{t}{{1 + t}}} \right)^{\frac{1}{2}} P_{\nu  - \frac{1}{2}}^{1} \left( {1+2t} \right)
\]
(cf. \cite[15.9.E21]{NIST}). The functional relation (see, e.g., \cite[14.10.E3]{NIST})
\[
\left( {\nu  - \frac{1}{2}  } \right)P_{\nu  + \frac{1}{2}}^{1  } \left( {1 + 2t} \right) - 2\nu \left( {1 + 2t} \right)P_{\nu  - \frac{1}{2}}^{1 } \left( {1 + 2t} \right) + \left( {\nu  + \frac{1}{2}  } \right)P_{\nu  - \frac{3}{2}}^{1 } \left( {1 + 2t} \right) = 0
\]
then implies
\begin{multline*}
{\bf F}\left( {\nu  + \frac{3}{2}, - \nu  - \frac{1}{2};0; - t} \right) + {\bf F}\left( {\nu  - \frac{1}{2}, - \nu  + \frac{3}{2};0; - t} \right) \\ =  \frac{1}{{2\nu }}\left( {{\bf F}\left( {\nu  + \frac{3}{2}, - \nu  - \frac{1}{2};0; - t} \right) - {\bf F}\left( {\nu  - \frac{1}{2}, - \nu  + \frac{3}{2};0; - t} \right)} \right) \\  + 2\left( {1 + 2t} \right){\bf F}\left( {\nu  + \frac{1}{2}, - \nu  + \frac{1}{2};0; - t} \right).
\end{multline*}
From \eqref{eq48}, it is seen that ${\bf F}\left( {\nu  + \frac{3}{2}, - \nu  - \frac{1}{2};0; - t} \right)$ and ${\bf F}\left( {\nu  + \frac{1}{2}, - \nu  + \frac{1}{2};0; - t} \right)$ are positive and ${\bf F}\left( {\nu  - \frac{1}{2}, - \nu  + \frac{3}{2};0; - t} \right)$ is negative whenever $t$ is positive and $\frac{1}{2} < \nu < \frac{3}{2}$. This completes the proof of the lemma.
\end{proof}

\begin{lemma} Let $N$ be a non-negative integer and let $\nu$ and $\lambda$ be arbitrary complex numbers such that $\Re\left(\lambda\right)>0$. Then
\begin{equation}\label{eq17}
a_N \left( \nu  \right) = \left( { - 1} \right)^N \frac{{\cos \left( {\pi \nu } \right)}}{\pi }\frac{{\Gamma \left( {N + \lambda } \right)}}{{2^N }}\int_0^{ + \infty } {\frac{{t^{\lambda  - 1} }}{{\left( {1 + t} \right)^{N + \lambda } }}{\bf F}\left( {\nu  + \frac{1}{2}, - \nu  + \frac{1}{2};\lambda ; - t} \right)dt} ,
\end{equation}
provided $\left|\Re\left(\nu\right)\right|<N+\frac{1}{2}$, and
\begin{gather}\label{eq18}
\begin{split}
b_N \left( \nu  \right) = \; & \left( { - 1} \right)^{N + 1} \frac{{\cos \left( {\pi \nu } \right)}}{\pi }\frac{{\Gamma \left( {N + \lambda } \right)}}{{2^{N + 1} }} \\ & \times \int_0^{ + \infty } {\frac{{t^{\lambda  - 1} }}{{\left( {1 + t} \right)^{N + \lambda } }}\left( {{\bf F}\left( {\nu  + \frac{3}{2}, - \nu  - \frac{1}{2};\lambda ; - t} \right) + {\bf F}\left( {\nu  - \frac{1}{2}, - \nu  + \frac{3}{2};\lambda ; - t} \right)} \right)dt} ,
\end{split}
\end{gather}
\begin{gather}\label{eq19}
\begin{split}
b_N \left( \nu  \right) = \; & \left( { - 1} \right)^{N + 1} \frac{{\cos \left( {\pi \nu } \right)}}{\pi }\frac{{\Gamma \left( N \right)}}{{2^{N + 1} }}  \\ & \times \left( {2 + \int_0^{ + \infty } {\frac{{t^{ - 1} }}{{\left( {1 + t} \right)^N }}\left( {{\bf F}\left( {\nu  + \frac{3}{2}, - \nu  - \frac{1}{2};0; - t} \right) + {\bf F}\left( {\nu  - \frac{1}{2}, - \nu  + \frac{3}{2};0; - t} \right)} \right)dt} } \right),
\end{split}
\end{gather}
provided that $\left|\Re\left(\nu\right)\right|<N-\frac{1}{2}$ and $N\geq 1$.
\end{lemma}

\begin{proof} From \eqref{eq70}, we can assert that
\begin{equation}\label{eq07}
a_N \left( \nu  \right) = z^N (R_{N + 1}^{\left( K \right)} \left( {z,\nu } \right) - R_N^{\left( K \right)} \left( {z,\nu } \right))
\end{equation}
for any $N\geq 0$. Substituting the integral representation \eqref{eq9} into the right-hand side of this equality and employing the relation $p \Lambda _{p + 1} \left( w \right) = w\left( {1 - \Lambda _p \left( w \right)} \right)$ (cf. \cite[8.8.E2]{NIST}), we arrive at \eqref{eq17}.

In order to prove \eqref{eq18}, we substitute \eqref{eq17} into the right-hand side of the relation
\begin{equation}\label{eq688}
2b_N \left( \nu  \right) = a_N \left( {\nu  + 1} \right) + a_N \left( {\nu  - 1} \right),
\end{equation}
which itself is a consequence of the connection formula $- 2K'_\nu  \left( z \right) = K_{\nu + 1} \left( z \right) + K_{\nu  - 1} \left( z \right)$.

Formula \eqref{eq19} follows from \eqref{eq45} by an argument similar to the derivation of \eqref{eq18} from \eqref{eq9}; the details are left to the reader.
\end{proof}

We continue with the proof of the bound \eqref{eq24}. Let $N$ be a non-negative integer and let $\nu$ and $\lambda$ be arbitrary real numbers such that $\left| {\nu} \right| < N + \frac{1}{2}$ and $\lambda > \max \left( {0,\frac{1}{2} - \left| \nu  \right|} \right)$, respectively. From \eqref{eq9}, using Lemma \ref{lemma2}, we infer that
\begin{align*}
\big| {R_N^{\left( K \right)} \left( {z,\nu } \right)} \big| \le \; & \frac{{\left| {\cos \left( {\pi \nu } \right)} \right|}}{\pi }\frac{{\Gamma \left( {N + \lambda } \right)}}{{2^N }}\frac{1}{{\left| z \right|^N }} \\ & \times \int_0^{ + \infty } {\frac{{t^{\lambda  - 1} }}{{\left( {1 + t} \right)^{N + \lambda } }}{\bf F}\left( {\nu  + \frac{1}{2}, - \nu  + \frac{1}{2};\lambda ; - t} \right)dt} \mathop {\sup }\limits_{r \geq 1} \big| {\Lambda _{N + \lambda } \left( {2zr } \right)} \big|,
\end{align*}
as long as $\left| {\arg z} \right| < \frac{3\pi}{2}$. The right-hand side can be simplified using Lemma \ref{lemma2} and the representation \eqref{eq17} for the coefficients $a_N \left( \nu  \right)$, and we deduce that
\[
\big| {R_N^{\left( K \right)} \left( {z,\nu } \right)} \big| \le \frac{{\left| {a_N \left( \nu  \right)} \right|}}{{\left| z \right|^N }}\mathop {\sup }\limits_{r \geq 1} \big| {\Lambda _{N + \lambda } \left( {2zr } \right)} \big|.
\]
Now the required bound \eqref{eq24} follows by letting $\lambda  \to \max \left( {0,\frac{1}{2} - \left| \nu  \right|} \right)$. The estimate \eqref{eq25} can be obtained analogously from \eqref{eq54}.

The bound \eqref{eq700} can be proved as follows. Assume that $N$ is a positive integer and $\nu$ is an arbitrary real number such that $\left| {\nu} \right| < N - \frac{1}{2}$. By substituting \eqref{eq45} into the functional relation \eqref{eq101}, we find
\begin{gather}\label{eq871}
\begin{split}
& R_N^{(K')} \left( {z,\nu } \right) = \left( { - 1} \right)^{N + 1} \frac{{\cos \left( {\pi \nu } \right)}}{\pi }\frac{{\Gamma \left( N \right)}}{{2^{N + 1} }}\frac{1}{{z^N }}
\\ & \times \left( 2\Lambda _{N } \left( {2z} \right) + \int_0^{ + \infty } \frac{{t^{ - 1} }}{{\left( {1 + t} \right)^N }}\left( {{\bf F}\left( {\nu  + \frac{3}{2}, - \nu  - \frac{1}{2};0; - t} \right) + {\bf F}\left( {\nu  - \frac{1}{2}, - \nu  + \frac{3}{2};0; - t} \right)} \right) \right. \\ & \times  \Bigg.\Lambda _{N } \left( {2z\left( {1 + t} \right)} \right)dt  \Bigg).
\end{split}
\end{gather}
By Lemma \ref{lemma3}, we can then assert that
\begin{align*}
& \big| {R_N^{(K')} \left( {z,\nu } \right)} \big| \le \frac{{\left| {\cos \left( {\pi \nu } \right)} \right|}}{\pi }\frac{{\Gamma \left( N \right)}}{{2^{N + 1} }}\frac{1}{{\left| z \right|^N }}
\\ & \times \left( {2 + \int_0^{ + \infty } {\frac{{t^{ - 1} }}{{\left( {1 + t} \right)^N }}\left( {{\bf F}\left( {\nu  + \frac{3}{2}, - \nu  - \frac{1}{2};0; - t} \right) + {\bf F}\left( {\nu  - \frac{1}{2}, - \nu  + \frac{3}{2};0; - t} \right)} \right)dt} } \right) \\ & \times \mathop {\sup }\limits_{r \geq 1} \big| {\Lambda _{N } \left( {2zr } \right)} \big|,
\end{align*}
provided $\left| {\arg z} \right| < \frac{3\pi}{2}$. The right-hand side of this inequality can be simplified using Lemma \ref{lemma3} and the representation \eqref{eq19} for the coefficients $b_N \left( \nu  \right)$ leading to the required estimate \eqref{eq700}. The bound \eqref{eq701} can be deduced similarly from \eqref{eq55} and \eqref{eq102}.

To prove formula \eqref{eq321}, we first note that $0 < \Lambda _p \left( w \right) < 1$ whenever $w>0$ and $p>0$ (see Proposition \ref{propb1}). Employing Lemma \ref{lemma3}, the mean value theorem of integration and the fact that the summands in the large parentheses in \eqref{eq871} have the same sign, we can infer that
\begin{align*}
& R_N^{(K')} \left( {z,\nu } \right) = \left( { - 1} \right)^{N + 1} \frac{{\cos \left( {\pi \nu } \right)}}{\pi }\frac{{\Gamma \left( N \right)}}{{2^{N + 1} }}\frac{1}{{z^N }}
\\ & \times \left( {2 + \int_0^{ + \infty } {\frac{{t^{ - 1} }}{{\left( {1 + t} \right)^N }}\left( {{\bf F}\left( {\nu  + \frac{3}{2}, - \nu  - \frac{1}{2};0; - t} \right) + {\bf F}\left( {\nu  - \frac{1}{2}, - \nu  + \frac{3}{2};0; - t} \right)} \right) dt} } \right)\theta _N^{(K')} \left( {z,\nu } \right),
\end{align*}
for some $0 < \theta _N^{(K')} \left( {z,\nu } \right) < 1$ depending on $z, \nu$ and $N$. Comparing this expression with the representation \eqref{eq19} for the coefficients $b_N \left( \nu  \right)$ yields the required formula \eqref{eq321} for $R_N^{(K')} \left( {z,\nu } \right)$. The corresponding result for $R_N^{(J')} \left( {z,\nu } \right)$ can be obtained using the analogue of the representation \eqref{eq871} and the fact that $0 < \Pi _p \left( w \right) < 1$ whenever $w>0$ and $p>0$ (see Proposition \ref{propb1}); we leave the details to the reader.

We now turn to the proof of the bound \eqref{eq01}. Assume that $N$ is a non-negative integer and $\nu$ is an arbitrary complex number satisfying $\left|\Re\left(\nu\right)\right|<N+\frac{1}{2}$. First note that since
\[
{\bf F}\left( {\nu  + \frac{1}{2}, - \nu  + \frac{1}{2};\frac{1}{2}; - t} \right) = \frac{{(\left( {1 + t} \right)^{\frac{1}{2}}  + t^{\frac{1}{2}} )^{2\nu }  + (\left( {1 + t} \right)^{\frac{1}{2}}  - t^{\frac{1}{2}} )^{2\nu } }}{{2\left( {1 + t} \right)^{\frac{1}{2}} }}
\]
for any $t>0$ (see, for instance, \cite[15.4.E13]{NIST}), the following inequality holds:
\begin{equation}\label{eq03}
\left| {{\bf F}\left( {\nu  + \frac{1}{2}, - \nu  + \frac{1}{2};\frac{1}{2}; - t} \right)} \right| \le {\bf F}\left( {\Re \left( \nu  \right) + \frac{1}{2}, - \Re \left( \nu  \right) + \frac{1}{2};\frac{1}{2}; - t} \right).
\end{equation}
Taking $\lambda=\frac{1}{2}$ in \eqref{eq9} and employing \eqref{eq03}, we find that
\begin{align*}
\big| {R_N^{\left( K \right)} \left( {z,\nu } \right)} \big| \le \; & \frac{{\left| {\cos \left( {\pi \nu } \right)} \right|}}{\pi }\frac{{\Gamma \left( {N + \frac{1}{2}} \right)}}{{2^N }}\frac{1}{{\left| z \right|^N }} \\ & \times \int_0^{ + \infty } {\frac{{t^{ - \frac{1}{2}} }}{{\left( {1 + t} \right)^{N + \frac{1}{2}} }}{\bf F}\left( {\Re \left( \nu  \right) + \frac{1}{2}, - \Re \left( \nu  \right) + \frac{1}{2};\frac{1}{2}; - t} \right)dt} \mathop {\sup }\limits_{r \geq 1} \big| {\Lambda _{N + \frac{1}{2}} \left( {2zr} \right)} \big|,
\end{align*}
as long as $\left| {\arg z} \right| < \frac{3\pi}{2}$. Now the required bound \eqref{eq01} can be obtained by simplifying the right-hand side of the above inequality using the representation \eqref{eq17} (with $\lambda=\frac{1}{2}$) for the coefficients $a_N \left( \nu  \right)$ (with $\Re \left( \nu  \right)$ in place of $\nu$). The estimate \eqref{eq02} can be deduced in essentially the same way from \eqref{eq54} with the specific choice of $\lambda=\frac{1}{2}$.

We close this section by proving the estimate \eqref{eq220}. For this purpose, suppose that $N$ is a positive integer and $\nu$ is an arbitrary complex number such that $\left|\Re\left(\nu\right)\right|<N-\frac{1}{2}$. By taking $\lambda=\frac{1}{2}$ in \eqref{eq9} and substituting the result into the functional relation \eqref{eq101}, we obtain
\begin{align*}
& R_N^{(K')} \left( {z,\nu } \right) = \left( { - 1} \right)^{N + 1} \frac{{\cos \left( {\pi \nu } \right)}}{\pi }\frac{{\Gamma \left( {N + \frac{1}{2}} \right)}}{{2^{N + 1} }}\frac{1}{{z^N }} \\ & \times \int_0^{ + \infty } {\frac{{t^{ - \frac{1}{2}} }}{{\left( {1 + t} \right)^{N + \frac{1}{2}} }}\left( {{\bf F}\left( {\nu  + \frac{3}{2}, - \nu  - \frac{1}{2};\frac{1}{2}; - t} \right) + {\bf F}\left( {\nu  - \frac{1}{2}, - \nu  + \frac{3}{2};\frac{1}{2}; - t} \right)} \right)\Lambda _{N + \frac{1}{2}} \left( {2z\left( {1 + t} \right)} \right)dt} .
\end{align*}
An application of the inequality \eqref{eq03} yields
\begin{align*}
& \big| {R_N^{(K')} \left( {z,\nu } \right)} \big| \le \frac{{\left| {\cos \left( {\pi \nu } \right)} \right|}}{\pi }\frac{{\Gamma \left( {N + \frac{1}{2}} \right)}}{{2^{N + 1} }}\frac{1}{{\left| z \right|^N }} \\ & \times \int_0^{ + \infty } {\frac{{t^{ - \frac{1}{2}} }}{{\left( {1 + t} \right)^{N + \frac{1}{2}} }}\left( {{\bf F}\left( {\Re \left( \nu  \right) + \frac{3}{2}, - \Re \left( \nu  \right) - \frac{1}{2};\frac{1}{2}; - t} \right) + {\bf F}\left( {\Re \left( \nu  \right) - \frac{1}{2}, - \Re \left( \nu  \right) + \frac{3}{2};\frac{1}{2}; - t} \right)} \right)dt} \\ & \times \mathop {\sup }\limits_{r \geq 1} \big| {\Lambda _{N + \frac{1}{2}} \left( {2zr } \right)} \big|,
\end{align*}
for $\left| {\arg z} \right| < \frac{3\pi}{2}$. We then simplify the right-hand side of this inequality using the representation \eqref{eq18} (with $\lambda=\frac{1}{2}$) for the coefficients $b_N \left( \nu  \right)$ (with $\Re \left( \nu  \right)$ in place of $\nu$) to arrive at the required estimate \eqref{eq220}. The bound \eqref{eq221} can be deduced in a similar manner from \eqref{eq54} and \eqref{eq102} with the value $\lambda=\frac{1}{2}$.

\section{Proof of the error bounds for the re-expansions}\label{section4}

In this section, we prove the bounds for the remainder terms $R_{N,M}^{\left( K \right)} \left( {z,\nu } \right)$, $R_{N,M}^{\left( J \right)} \left( {z,\nu } \right)$, $R_{N,M}^{(K')} \left( {z,\nu } \right)$ and $R_{N,M}^{(J')} \left( {z,\nu } \right)$ that are given in Theorems \ref{theorem10}--\ref{theorem13}. We begin by stating and proving the following lemma.

\begin{lemma} Let $N$ be a non-negative integer and let $\nu$ be an arbitrary complex number. Then
\begin{equation}\label{eq732}
a_N \left( \nu  \right) = \left( { - 1} \right)^N \left( {\frac{2}{\pi }} \right)^{\frac{1}{2}} \frac{{\cos \left( {\pi \nu } \right)}}{\pi }\int_0^{ + \infty } {t^{N - \frac{1}{2}} e^{ - t} K_\nu  \left( t \right)dt} ,
\end{equation}
provided $\left|\Re\left(\nu\right)\right|<N+\frac{1}{2}$, and
\begin{equation}\label{eq733}
b_N \left( \nu  \right) = \left( { - 1} \right)^N \left( {\frac{2}{\pi }} \right)^{\frac{1}{2}} \frac{{\cos \left( {\pi \nu } \right)}}{\pi }\int_0^{ + \infty } {t^{N - \frac{1}{2}} e^{ - t} K'_\nu  \left( t \right)dt} ,
\end{equation}
provided that $\left|\Re\left(\nu\right)\right|<N-\frac{1}{2}$ and $N \geq 1$.
\end{lemma}

\begin{proof}
Formula \eqref{eq732} can be derived by substituting the integral representation \eqref{eq1} into the right-hand side of the equality \eqref{eq07}. To prove \eqref{eq733}, we substitute \eqref{eq732} into \eqref{eq688} and use the functional relation $- 2K'_\nu  \left( t \right) = K_{\nu + 1} \left( t \right) + K_{\nu  - 1} \left( t \right)$.
\end{proof}

We continue with the proof of the estimate \eqref{eq08}. Let $N$ and $M$ be arbitrary non-negative integers satisfying $M < N$. Suppose that $\left|\Re\left(\nu\right)\right| < M+ \frac{1}{2}$ and $\left|\arg z\right|<\pi$. We begin by replacing the function $K_\nu\left(t\right)$ in the integral formula \eqref{eq1} for $R_N^{\left( K \right)} \left( {z,\nu } \right)$ by its truncated asymptotic expansion 
\begin{equation}\label{eq171}
K_\nu  \left( t \right) = \left( {\frac{\pi }{{2t}}} \right)^{\frac{1}{2}} e^{ - t} \left( {\sum\limits_{m = 0}^{M - 1} {\frac{{a_m \left( \nu  \right)}}{{t^m }}}  + R_M^{\left( K \right)} \left( {t,\nu } \right)} \right)
\end{equation}
and using the representation \eqref{eq355} of the basic terminant $\Lambda_p\left(w\right)$. In this way, we obtain \eqref{eq437} with
\begin{gather}\label{eq06}
\begin{split}
R_{N,M}^{\left( K \right)} \left( {z,\nu } \right) & = \left( { - 1} \right)^N \frac{{\cos \left( {\pi \nu } \right)}}{\pi }\frac{1}{{z^N }}\int_0^{ + \infty } {\frac{{t^{N - 1} e^{ - 2t} }}{{1 + t/z}}R_M^{\left( K \right)} \left( {t,\nu } \right)dt} 
\\ & = \left( { - 1} \right)^N \frac{{\cos \left( {\pi \nu } \right)}}{\pi }e^{ - iN\arg z} \int_0^{ + \infty } {\frac{{u^{N - 1} e^{ - 2\left| z \right|u} }}{{1 + ue^{ - i\arg z} }}R_M^{\left( K \right)} \left( {\left| z \right|u,\nu } \right)du} .
\end{split}
\end{gather}
When passing to the second equality, we have made a change of integration variable from $t$ to $u$ by $t=\left|z\right|u$. The remainder $R_M^{\left( K \right)} \left( {\left| z \right|u,\nu } \right)$ is given by the integral formula \eqref{eq1}, which can be re-expressed in the form
\begin{gather}\label{eq09}
\begin{split}
R_M^{\left( K \right)} \left( {\left|z\right|u ,\nu } \right) = \; & \left( { - 1} \right)^M \left( {\frac{2}{\pi }} \right)^{\frac{1}{2}} \frac{{\cos \left( {\pi \nu } \right)}}{\pi }\frac{1}{{\left( {\left|z\right|u } \right)^M }}\int_0^{ + \infty } {\frac{{t^{M - \frac{1}{2}} e^{ - t} }}{{1 + t/\left|z\right|}}K_\nu  \left( t \right)dt} 
\\ & + \left( { - 1} \right)^M \left( {\frac{2}{\pi }} \right)^{\frac{1}{2}} \frac{{\cos \left( {\pi \nu } \right)}}{\pi }\frac{{u  - 1}}{{\left( {\left|z\right|u } \right)^M }}\int_0^{ + \infty } {\frac{{t^{M - \frac{1}{2}} e^{ - t} }}{{\left( {1 + \left|z\right|u /t} \right)\left( {1 + t/\left|z\right|} \right)}}K_\nu  \left( t \right)dt} 
\\ = \; & u^{ - M} R_M^{\left( K \right)} \left( {\left| z \right|,\nu } \right) \\ & + \left( { - 1} \right)^M \left( {\frac{2}{\pi }} \right)^{\frac{1}{2}} \frac{{\cos \left( {\pi \nu } \right)}}{\pi }\frac{{u  - 1}}{{\left( {\left|z\right|u } \right)^M }}\int_0^{ + \infty } {\frac{{t^{M - \frac{1}{2}} e^{ - t} }}{{\left( {1 + \left|z\right|u /t} \right)\left( {1 +t/\left|z\right|} \right)}}K_\nu  \left( t \right)dt} .
\end{split}
\end{gather}
Noting that
\[
0 < \frac{1}{{\left( {1 + \left| z \right|u/t} \right)\left( {1 + t/\left| z \right|} \right)}} < 1
\]
for positive $u$ and $t$, the substitution of \eqref{eq09} into \eqref{eq06} and trivial estimation yield the upper bound
\begin{gather}\label{eq99}
\begin{split}
\big| {R_{N,M}^{\left( K \right)} \left( {z,\nu } \right)} \big| \le \; & \frac{{\left| {\cos \left( {\pi \nu } \right)} \right|}}{\pi }\left| {\int_0^{ + \infty } {\frac{{u^{N - M - 1} e^{ - 2\left|z\right|u } }}{{1 + u e^{ - i\arg z} }}du } } \right|\big| {R_M^{\left( K \right)} \left( {\left| z \right|,\nu } \right)} \big|
\\ & + \frac{{\left| {\cos \left( {\pi \nu } \right)} \right|}}{\pi }\frac{{\left| {\cos \left( {\pi \nu } \right)} \right|}}{{\left| {\cos \left( {\pi \Re \left( \nu  \right)} \right)} \right|}}\left| {a_M \left( {\Re \left( \nu  \right)} \right)} \right|\frac{1}{{\left| z \right|^M }}\int_0^{ + \infty } {u ^{N - M - 1} e^{ - 2\left|z\right|u } \left| {\frac{{u  - 1}}{{u  + e^{i\arg z } }}} \right|du } .
\end{split}
\end{gather}
In arriving at this bound, we have made use of the fact that $\left|K_\nu\left(t\right)\right|<K_{\Re\left(\nu\right)}\left(t\right)$ for any $t > 0$ and of the representation \eqref{eq732} for the coefficients $a_N \left( \nu  \right)$ (with $M$ and $\Re \left( \nu  \right)$ in place of $N$ and $\nu$). Since $\left| {\left( {u - 1} \right)/\left( {u + e^{i\arg z} } \right)} \right| \le 1$ for positive $u$, we find, after simplification, that
\begin{align*}
\big| {R_{N,M}^{\left( K \right)} \left( {z,\nu } \right)} \big| \le \; & \frac{{\left| {\cos \left( {\pi \nu } \right)} \right|}}{{2^N \pi }}\frac{1}{{\left| z \right|^N }}2^M \left| z \right|^M \big| {R_M^{\left( K \right)} \left( {\left|z\right|,\nu } \right)} \big| \Gamma \left( {N - M} \right) \left| {\Lambda _{N - M  } \left( {2z} \right)} \right| \\ & + \frac{{\left| {\cos \left( {\pi \nu } \right)} \right|}}{{2^N \pi }}\frac{1}{{\left| z \right|^N }}2^M \frac{{\left| {\cos \left( {\pi \nu } \right)} \right|}}{{\left| {\cos \left( {\pi \Re \left( \nu  \right)} \right)} \right|}}\left| {a_M \left( {\Re \left( \nu  \right)} \right)} \right|\Gamma \left( {N - M} \right).
\end{align*}
By continuity, this bound holds in the closed sector $\left|\arg z\right| \leq \pi$, and therefore the proof of the estimate \eqref{eq08} is complete.

To prove the bound \eqref{eq822}, we proceed in a similar manner. Let $N$ and $M$ be arbitrary non-negative integers satisfying $M < N$. Assume further that $\left|\Re\left(\nu\right)\right| < M+ \frac{1}{2}$ and $\left|\arg z\right|<\frac{\pi}{2}$. We replace the function $K_\nu\left(t\right)$ in the integral formula \eqref{eq91} for $R_N^{\left( J \right)} \left( {z,\nu } \right)$ by its truncated asymptotic expansion \eqref{eq171} and employ the representation \eqref{eq356} of the basic terminant $\Pi_p\left(w\right)$, to obtain \eqref{eq737} with
\begin{align*}
R_{N,M}^{\left( J \right)} \left( {z,\nu } \right) & = \left( { - 1} \right)^{\left\lfloor {N/2} \right\rfloor } \frac{{\cos \left( {\pi \nu } \right)}}{\pi }\frac{1}{{z^N }}\int_0^{ + \infty } {\frac{{t^{N - 1} e^{ - 2t} }}{{1 + \left( {t/z} \right)^2 }}R_M^{\left( K \right)} \left( {t,\nu } \right)dt} 
\\ & = \left( { - 1} \right)^{\left\lfloor {N/2} \right\rfloor } \frac{{\cos \left( {\pi \nu } \right)}}{\pi }e^{ - i N \arg z } \int_0^{ + \infty } {\frac{{u ^{N - 1} e^{ - 2\left|z\right|u } }}{{1 +u ^2 e^{ - 2i\arg z } }}R_M^{\left( K \right)} \left( {\left|z\right|u ,\nu } \right)du } .
\end{align*}
By an argument similar to that which led to \eqref{eq99}, we deduce the bound
\begin{align*}
\big| {R_{N,M}^{\left( J \right)} \left( {z,\nu } \right)} \big| \le \; & \frac{{\left| {\cos \left( {\pi \nu } \right)} \right|}}{\pi }\left| {\int_0^{ + \infty } {\frac{{u^{N - M - 1} e^{ - 2\left|z\right|u } }}{{1 + u^2 e^{ - 2i\arg z} }}du } } \right|\big| {R_M^{\left( K \right)} \left( {\left| z \right|,\nu } \right)} \big|
\\ & + \frac{{\left| {\cos \left( {\pi \nu } \right)} \right|}}{\pi }\frac{{\left| {\cos \left( {\pi \nu } \right)} \right|}}{{\left| {\cos \left( {\pi \Re \left( \nu  \right)} \right)} \right|}}\left| {a_M \left( {\Re \left( \nu  \right)} \right)} \right|\frac{1}{{\left| z \right|^M }}\int_0^{ + \infty } {u^{N - M - 1} e^{ - 2\left|z\right|u} \left| {\frac{{u  - 1}}{{u^2  + e^{2i\arg z } }}} \right|du } .
\end{align*}
Now, using the fact that $\left| {\left( {u - 1} \right)/\left( {u^2 + e^{2i\arg z} } \right)} \right| \le 1$ for any positive $u$, after simplification we arrive at
\begin{align*}
\big| {R_{N,M}^{\left( J \right)} \left( {z,\nu } \right)} \big| \le \; & \frac{{\left| {\cos \left( {\pi \nu } \right)} \right|}}{{2^N \pi }}\frac{1}{{\left| z \right|^N }}2^M \left| z \right|^M \big| {R_M^{\left( K \right)} \left( {\left| z \right|,\nu } \right)} \big|\Gamma \left( {N - M} \right)\left| {\Pi _{N - M} \left( {2z} \right)} \right| \\ & + \frac{{\left| {\cos \left( {\pi \nu } \right)} \right|}}{{2^N \pi }}\frac{1}{{\left| z \right|^N }}2^M \frac{{\left| {\cos \left( {\pi \nu } \right)} \right|}}{{\left| {\cos \left( {\pi \Re \left( \nu  \right)} \right)} \right|}}\left| {a_M \left( {\Re \left( \nu  \right)} \right)} \right|\Gamma \left( {N - M} \right).
\end{align*}
By continuity, this bound holds in the closed sector $\left|\arg z\right| \leq \frac{\pi}{2}$, and thus the proof of the estimate \eqref{eq822} is complete.

The corresponding bounds \eqref{eq841} and \eqref{eq842} for the remainder terms $R_{N,M}^{(K')} \left( {z,\nu } \right)$ and $R_{N,M}^{(J')} \left( {z,\nu } \right)$ can be proved in an analogous manner using the representations \eqref{eq843}, \eqref{eq92}, \eqref{eq93} and \eqref{eq733}. We leave the details to the reader.

We now turn to the proof of the expression \eqref{eq575}. Let $N$ and $M$ be arbitrary non-negative integers satisfying $M < N$. Suppose that $z$ is positive and $\nu$ is a real number such that $\left|\nu\right| < M+ \frac{1}{2}$. Applying the first formula in \eqref{eq711} (with $M$ and $t$ in place of $N$ and $z$) in the representation \eqref{eq06} for $R_{N,M}^{\left( K \right)} \left( {z,\nu } \right)$, and then using the mean value theorem of integration, we find that
\begin{align*}
R_{N,M}^{\left( K \right)} \left( {z,\nu } \right) & = \left( { - 1} \right)^N \frac{{\cos \left( {\pi \nu } \right)}}{\pi }\frac{1}{{z^N }}a_M \left( \nu  \right)\int_0^{ + \infty } {\frac{{t^{N - M - 1} e^{ - 2t} }}{{1 + z/t}}\theta _M^{\left( K \right)} \left( {t,\nu } \right)dt} 
\\ & = \left( { - 1} \right)^N \frac{{\cos \left( {\pi \nu } \right)}}{\pi }\frac{1}{{z^N }}a_M \left( \nu  \right)\Theta _{N,M}^{\left( K \right)} \left( {z,\nu } \right)\int_0^{ + \infty } {\frac{{t^{N - M - 1} e^{ - 2t} }}{{1 + z/t}}dt} .
\end{align*}
Here, $0<\Theta _{N,M}^{\left( K \right)} \left( {z,\nu } \right)<1$ is a suitable number that depends on $z, \nu$, $N$ and $M$. The integral in the second line can be expressed in terms of the basic terminant $\Lambda_p\left(w\right)$ by making use of the formula \eqref{eq355}. Hence the expression \eqref{eq575} follows.

The corresponding results \eqref{eq576}--\eqref{eq577} for $R_{N,M}^{\left( J \right)} \left( {z,\nu } \right)$, $R_{N,M}^{(K')} \left( {z,\nu } \right)$ and $R_{N,M}^{(J')} \left( {z,\nu } \right)$ can be obtained using the formulae \eqref{eq711}, \eqref{eq321} and the analogues of the representation \eqref{eq06}.

\section{Discussion}\label{section5}

In this paper, we have derived new integral representations and estimates for the remainder terms of the large-argument asymptotic expansions of the Hankel, Bessel and modified Bessel functions, and their derivatives. We have also constructed error bounds for the re-expansions of these remainders. In this section, we shall discuss the sharpness of our error bounds and their relation to other results in the literature.

First, we show that the error bounds proved in this paper are reasonably sharp. For the sake of brevity, we consider only the bounds for ${R_N^{\left( K \right)} \left( {z,\nu } \right)}$ and $R_{N,M}^{\left( K \right)} \left( {z,\nu } \right)$; the other remainder terms can be treated in a similar manner.

Let $N$ be any non-negative integer, $\nu$ a real number and $z$ a complex number. Suppose that $\left|\nu\right|<N+\frac{1}{2}$ and $\left|\arg z\right|\leq \frac{\pi}{2}$. Under these assumptions, it follows from Theorem \ref{theorem3} and Proposition \ref{propb1} that
\begin{equation}\label{eq558}
\big| {R_N^{\left( K \right)} \left( {z,\nu } \right)} \big| \le \frac{{\left| {a_N \left( \nu  \right)} \right|}}{{\left| z \right|^N }}.
\end{equation}
By the definition of an asymptotic expansion, $\lim _{z \to \infty } \left| z \right|^N \big| {R_N^{\left( K \right)} \left( {z,\nu } \right)} \big| = \left| {a_N \left( \nu  \right)} \right|$ for any fixed $N\geq 0$. Therefore, when $\left|\arg z\right|\leq \frac{\pi}{2}$, the estimate \eqref{eq558} and hence our error bound \eqref{eq24} cannot be improved in general.

Consider now the case when $\frac{\pi}{2}<\left|\arg z\right|\leq \pi$ ($\nu$ is still real and satisfies $\left|\nu\right|<N+\frac{1}{2}$). Combining Theorem \ref{theorem3} (or Theorem \ref{theorem6}) with Propositions \ref{propb1} and \ref{propb3}, we deduce that
\begin{equation}\label{eq662}
\big| {R_N^{\left( K \right)} \left( {z,\nu } \right)} \big| \le \frac{{\left| {a_N \left( \nu  \right)} \right|}}{{\left| z \right|^N }} \times
\min \left( {\left| {\csc \left( {\arg z} \right)} \right|,1 + \chi \left( {N + \frac{1}{2}} \right)} \right).
\end{equation}
Here, following Olver \cite{Olver1}, we use the notation
\begin{equation}\label{eq611}
\chi \left( p \right) = \pi ^{\frac{1}{2}} \frac{{\Gamma \left( {\frac{p}{2} + 1} \right)}}{{\Gamma \left( {\frac{p}{2} + \frac{1}{2}} \right)}},
\end{equation}
for any $p > 0$. The bound \eqref{eq662} is reasonably sharp as long as $\left| {\csc \left( {\arg z} \right)} \right|$ is not very large, i.e., when $\left|\arg z\right|$ is bounded away from $\pi$. As $\left|\arg z\right|$ approaches $\pi$, the factor $\left| {\csc \left( {\arg z} \right)} \right|$ grows indefinitely and therefore it has to be replaced by $1 + \chi \left( {N + \frac{1}{2}} \right)$. By Stirling's formula, $\chi \left( {N + \frac{1}{2}} \right) \sim \left( {\frac{{\pi N}}{2}} \right)^{\frac{1}{2}}$ as $N \to +\infty$, and therefore the appearance of this factor in the bound \eqref{eq662} may give the impression that this estimate is unrealistic for large $N$. However, this is not the case, as the following argument shows. We may suppose, without loss of generality, that $2\nu$ is not equal to an odd integer, because otherwise, the remainder term ${R_N^{\left( K \right)} \left( {z,\nu } \right)}$ becomes identically zero, since $\left|\nu\right|<N+\frac{1}{2}$. Recently, Paris \cite{Paris} showed that if $a$ and $a-b+1$ are fixed and neither of them is a non-positive integer, then the confluent hypergeometric function satisfies
\begin{gather}\label{eq044}
\begin{split}
U(a,b,xe^{ \pm \pi i} ) = \; & \frac{{(xe^{ \pm \pi i} )^{ - a} }}{{\Gamma \left( a \right)\Gamma \left( {a - b + 1} \right)}}\sum\limits_{n = 0}^{N - 1} {\frac{{\Gamma \left( {a + n} \right)\Gamma \left( {a - b + 1 + n} \right)}}{{\Gamma \left( {n + 1} \right)x^n }}} \\ & \pm \frac{{\pi ie^{ \mp \pi ia} x^{a - b} e^{ - x} }}{{\Gamma \left( a \right)\Gamma \left( {a - b + 1} \right)}}\left( {1 + \mathcal{O}_{a,b}\left( {\frac{1}{{x^{1/2} }}} \right)} \right)
\end{split}
\end{gather}
as $x\to +\infty$, provided that $\left|N - x\right|$ remains bounded. In the following, we assume that $\nu$ is fixed, $\left|N - 2x\right|$ is bounded and $x$ is large and positive. Employing the expansion \eqref{eq044} for the right-hand side of the functional relation
\begin{equation}\label{eq092}
K_\nu  \left( z \right) = \pi^{\frac{1}{2}} e^{ - z} \left( {2z} \right)^\nu  U\left( \nu  + \tfrac{1}{2},2\nu  + 1,2z \right)
\end{equation}
(see, e.g., \cite[13.6.E10]{NIST}) and comparing the result with \eqref{eq70}, we obtain
\begin{equation}\label{eq091}
\big| {R_N^{\left( K \right)} (xe^{ \pm \pi i} ,\nu )} \big| = \left| {\cos \left( {\pi \nu } \right)} \right|e^{ - 2x} \left( {1 + \mathcal{O}_{\nu}\left( {\frac{1}{{x^{1/2} }}} \right)} \right).
\end{equation}
From \eqref{eq662}, we have
\begin{equation}\label{eq090}
\big| {R_N^{\left( K \right)} (xe^{ \pm \pi i} ,\nu )} \big| \le \left( {1 + \chi \left( {N + \frac{1}{2}} \right)} \right)\frac{{\left| {a_N \left( \nu  \right)} \right|}}{{x^N }}.
\end{equation}
Now, using our assumptions on $N$, $\nu$ and $z$, we have, by Stirling's formula,
\begin{equation}\label{eq088}
\left| {a_N \left( \nu  \right)} \right| = \frac{{\left| {\cos \left( {\pi \nu } \right)} \right|}}{\pi }\left( {\frac{N}{{2e}}} \right)^N \left( {\frac{{2\pi }}{N}} \right)^{\frac{1}{2}} \left( {1 + \mathcal{O}_\nu  \left( {\frac{1}{N}} \right)} \right) = \frac{{\left| {\cos \left( {\pi \nu } \right)} \right|}}{{\pi ^{\frac{1}{2}} }}x^{N - \frac{1}{2}} e^{ - 2x} \left( {1 + \mathcal{O}_\nu  \left( {\frac{1}{x}} \right)} \right).
\end{equation}
Similarly,
\begin{gather}\label{eq089}
\begin{split}
1 + \chi \left( {N + \frac{1}{2}} \right) & = 1 + \left( {\frac{{\pi N}}{2}} \right)^{\frac{1}{2}}  \left( {1 + \mathcal{O}\left( {\frac{1}{N}} \right)} \right) \\ & = \left( {\frac{{\pi N}}{2}} \right)^{\frac{1}{2}} \left( {1 + \mathcal{O}\left( {\frac{1}{{N^{1/2} }}} \right)} \right) = \left( {\pi x} \right)^{\frac{1}{2}} \left( {1 + \mathcal{O}\left( {\frac{1}{{x^{1/2} }}} \right)} \right).
\end{split}
\end{gather}
Combining \eqref{eq088} and \eqref{eq089}, we see that the upper bound in \eqref{eq090} is asymptotically equal to the right-hand side of the equality \eqref{eq091}. Consequently, when $\left|\arg z\right|$ is equal or close to $\pi$, the estimate \eqref{eq662} and thus the error bound \eqref{eq24} cannot be improved in general.

Finally, assume that $\pi  < \left| {\arg z} \right| < \frac{3\pi}{2}$, $\nu$ is real and $\left|\nu\right|<N+\frac{1}{2}$. If we combine Theorem \ref{theorem3} (or Theorem \ref{theorem6}) with Propositions \ref{propb3} and \ref{propb4}, we obtain the estimate
\begin{equation}\label{eq023}
\big| {R_N^{\left( K \right)} \left( {z,\nu } \right)} \big| \le \frac{{\left| {a_N \left( \nu  \right)} \right|}}{{\left| z \right|^N }} 
\left( {\frac{{\sqrt {\pi \left( {2N + 1} \right)} }}{{\left| {\cos \left( {\arg z} \right)} \right|^{N + \frac{1}{2}} }} + 1+\chi \left( {N + \frac{1}{2}} \right) } \right).
\end{equation}
Elementary analysis shows that the second factor on the right-hand side of \eqref{eq023}, as a function of $N$, remains bounded, provided that $\left| {\arg z} \right| - \pi  = \mathcal{O}\big( {N^{ - \frac{1}{2}} } \big)$. This gives a reasonable estimate for the remainder term $R_N^{\left( K \right)} \left( {z,\nu } \right)$. Otherwise, $\left| {\cos \left( {\arg z} \right)} \right|^{N + \frac{1}{2}}$ can take very small values when $N$ is large, making the bound \eqref{eq023} completely unrealistic in most of the sectors $\pi  < \left| {\arg z} \right| < \frac{3\pi}{2}$. This deficiency of the bound \eqref{eq023} (and \eqref{eq24}) is necessary and is due to the omission of certain exponentially small terms arising from the Stokes phenomenon related to the asymptotic expansion \eqref{eq88} of the modified Bessel function (for a detailed discussion, see \cite{Olver4}). Thus, the use of the asymptotic expansion \eqref{eq88} should be confined to the sector $\left|\arg z\right|\leq \pi$. For other ranges of $\arg z$, one should use the analytic continuation formulae for the modified Bessel function \cite[10.34.E3]{NIST}.

The bound \eqref{eq24} and the bound \eqref{eq01} which we discuss below and which covers the case of complex $\nu$ differ significantly only in the component that depends on $\nu$. Therefore, when analyzing the estimate \eqref{eq01} we can focus mainly on understanding the relationship between the quantities $\frac{{\left| {\cos \left( {\pi \nu } \right)} \right|}}{{\left| {\cos \left( {\pi \Re \left( \nu  \right)} \right)} \right|}}\left| {a_N \left( {\Re \left( \nu  \right)} \right)} \right|$ and $\left| {a_N \left( \nu  \right)} \right|$. With the aid of the definition \eqref{eq612}, one readily infers that
\begin{equation}\label{eq483}
\frac{{\left| {\cos \left( {\pi \nu } \right)} \right|}}{{\left| {\cos \left( {\pi \Re \left( \nu  \right)} \right)} \right|}}\left| {a_N \left( {\Re \left( \nu  \right)} \right)} \right| = \frac{{\Gamma \left( {N + \frac{1}{2} + \Re \left( \nu  \right)} \right)\Gamma \left( {N + \frac{1}{2} - \Re \left( \nu \right)} \right)}}{{\left| {\Gamma \left( {N + \frac{1}{2} + \nu } \right)\Gamma \left( {N + \frac{1}{2} - \nu } \right)} \right|}}\left| {a_N \left( \nu  \right)} \right|.
\end{equation}
Now, we make the assumptions that $N - \left| {\Re \left( \nu  \right)} \right| \to  + \infty$ and $\left| {\Im \left( \nu  \right)} \right| = o(N^{\frac{1}{2}} )$ as $N\to +\infty$. With these provisos, it can easily be shown, using for example Stirling's formula, that the quotient of gamma functions in \eqref{eq483} is asymptotically $1$ for large $N$. Consequently, the right-hand side of the inequality \eqref{eq01} is asymptotic to
\[
\frac{{\left| {a_N \left( \nu  \right)} \right|}}{{\left| z \right|^N }}\mathop {\sup }\limits_{r \ge 1} \big| {\Lambda _{N + \frac{1}{2}} \left( {2zr} \right)} \big|
\]
if $N$ is large. Now, using the same argument as in the case of real $\nu$ above, we can conclude that the bound \eqref{eq01} is sharp in general (at least when $z$ is restricted to the sector $\left|\arg z\right|\leq \pi$). If the assumption $\left| {\Im \left( \nu  \right)} \right| = o(N^{\frac{1}{2}} )$ is replaced by the weaker condition $\left| {\Im \left( \nu  \right)} \right| = \mathcal{O}(N^{\frac{1}{2}} )$, the quotient in \eqref{eq483} is still bounded and hence, the estimate \eqref{eq01} is still reasonably sharp. Otherwise, if $N - \left| {\Re \left( \nu  \right)} \right|$ is small and $\left| \Im\left(\nu\right)  \right|$ is much larger than $N^{\frac{1}{2}}$, this quotient may grow exponentially fast in $\left| \Im\left(\nu\right)  \right|$, which can make the bound \eqref{eq01} completely unrealistic. Note, however, that in practical applications, it is reasonable to truncate the asymptotic expansion \eqref{eq88} just before its numerically least term, i.e., when $N \approx 2\left| z \right|$. And so, since the expansion \eqref{eq88} is valid only when $\nu ^2  = o\left( {\left| z \right|} \right)$, the conditions $N - \left| {\Re \left( \nu  \right)} \right| \to  + \infty$ and $\left| {\Im \left( \nu  \right)} \right| = o(N^{\frac{1}{2}} )$ are automatically satisfied in this case.

Let us now turn to the analysis of the estimate \eqref{eq08} for $R_{N,M}^{\left( K \right)} \left( {z,\nu } \right)$. When $z$ is large, $\left| z \right|^M \big| {R_M^{\left( K \right)} \left( {\left| z \right|,\nu } \right)} \big| \lesssim \left| {a_M \left( \nu  \right)} \right|$ holds, and therefore the first term on the right-hand side of the inequality \eqref{eq08} is of the same order of magnitude as the first neglected term in the expansion \eqref{eq437}. It can be shown that, when $N - M$ is large, the second term is comparable with, or less than, the first term (except near the zeros of $\Lambda_{N-M}\left(2z\right)$). The proof of this fact is identical to the proof given by Boyd \cite{Boyd} for the case when $\nu$ is real and $\left|\nu\right|<\frac{1}{2}$ and it is therefore not pursued here. In summary, the bound \eqref{eq08} is comparable with the first neglected term in the expansion \eqref{eq437}, unless $z$ is close to a zero of $\Lambda_{N-M}\left(2z\right)$, and thus, this bound is reasonably sharp. We would like to remark that an estimate for $R_{N,M}^{\left( K \right)} \left( {z,\nu } \right)$, different from \eqref{eq08}, can be derived using the functional relation \eqref{eq092} and an error bound for the re-expansion of the remainder term in the large-$z$ asymptotic expansion of the confluent hypergeometric function $U\left(a,b,z\right)$ which is due to Olver \cite{Olver3}. Since the resulting estimate is significantly more complicated than \eqref{eq08}, we omit the details.

There are two major methods in the literature for obtaining bounds for the remainder terms $R_N^{\left( K \right)} \left( {z,\nu } \right)$ and $R_N^{\left( J \right)} \left( {z,\nu } \right)$: one is based on integral representations, while the other uses differential equations. The former approach was taken by Schl\"{a}fli \cite{Schlafli}, Watson \cite[pp. 209--210]{Watson} and Meijer \cite{Meijer}. The bounds proved by Schl\"{a}fli and Meijer were later examined by D\"{o}ring \cite{Doring}, in the case of real $\nu$, who effected simplifications to make them more easily computable. These results of the authors can all be deduced as direct consequences of Theorem \ref{theorem6} and Propositions \ref{propb1} and \ref{propb2}.

Weber \cite{Weber} and Olver \cite{Olver1} used differential equation methods to derive different types of estimates for the remainder $R_N^{\left( K \right)} \left( {z,\nu } \right)$. Since Olver's bounds are more general and sharper than those found by Weber, we consider only Olver's estimates below. First, we state his results and then compare them with our estimates.

Let $z$ and $\nu$ be arbitrary complex numbers such that $\left|\arg z\right| < \frac{3\pi}{2}$. Then Olver's bounds read as follows:
\begin{equation}\label{eq776}
\big| {R_0^{\left( K \right)} \left( {z,\nu } \right)} \big| \le \exp \left( {\left| {\nu ^2  - \tfrac{1}{4}} \right|\mathscr{V}_{z,\infty } \left( {t^{ - 1} } \right)} \right)
\end{equation}
(cf. \cite[exer. 13.4, p. 270]{Olver5}) and
\begin{equation}\label{eq777}
\big| {R_N^{\left( K \right)} \left( {z,\nu } \right)} \big| \le 2\left| {a_N \left( \nu  \right)} \right|\mathscr{V}_{z,\infty } \left( {t^{ - N} } \right)\exp \left( {\left| {\nu ^2  - \tfrac{1}{4}} \right|\mathscr{V}_{z,\infty } \left( {t^{ - 1} } \right)} \right)
\end{equation}
for $N\geq 1$ (cf. \cite[10.40.E11]{NIST} or \cite[eq. 13.02, p. 267]{Olver5}). In addition, when $z$ is positive and $\nu$ is real, the factor of $2$ on the right-hand side of \eqref{eq777} can be omitted (cf. \cite[Remark (c), p. 266]{Olver5}). In \eqref{eq776} and \eqref{eq777}, $\mathscr{V}$ denotes the variational operator, in particular $\mathscr{V}_{z,\infty } \left( {t^{ - N} } \right)=N\int_z^\infty  {\left| t \right|^{ - N - 1} \left| {dt} \right|}$ and the path of variation is subject to the condition that $\left|\Re\left(t\right)\right|$ changes monotonically. The following are simple bounds for $\mathscr{V}_{z,\infty } \left( {t^{ - N} } \right)$:
\begin{equation}\label{eq778}
\mathscr{V}_{z,\infty } \left( {t^{ - N} } \right) \le \begin{cases} \left|z\right|^{-N} & \text{ if } \; \left|\arg z\right| \leq \frac{\pi}{2}, \\ \chi\left(N\right)\left|z\right|^{-N} & \text{ if } \; \frac{\pi}{2} < \left|\arg z\right| \leq \pi, \\ 2\chi\left(N\right)\left|z \cos\left(\arg z\right)\right|^{-N} & \text{ if } \; \pi < \left|\arg z\right| < \frac{3\pi}{2},\end{cases}
\end{equation}
where $\chi\left(N\right)$ is given by \eqref{eq611} with $p=N$ (see, for instance, \cite[10.40.E12]{NIST} or \cite[pp. 224--227]{Olver5}).

Corresponding estimates for the error term $R_N^{\left( J \right)} \left( {z,\nu } \right)$ can be obtained by combining \eqref{eq776} and \eqref{eq777} with the inequality
\[
\big| {R_N^{\left( J \right)} \left( {z,\nu } \right)} \big| \le \frac{1}{2}\left( {\big| {R_N^{\left( K \right)} \left( {ze^{\frac{\pi }{2}i} ,\nu } \right)} \big| + \big| {R_N^{\left( K \right)} \left( {ze^{ - \frac{\pi }{2}i} ,\nu } \right)} \big|} \right),
\]
which follows from \eqref{eq56} and \eqref{eq57}.

Now we turn to the comparison of Olver's bounds with ours. Note that, unlike our error bounds, Olver's results hold without any restrictions on $\left| {\Re \left( \nu  \right)} \right|$. For the purpose of further comparison, we first combine \eqref{eq777} with \eqref{eq778} to obtain
\begin{equation}\label{eq779}
\big| {R_N^{\left( K \right)} \left( {z,\nu } \right)} \big| \le \frac{\left| {a_N \left( \nu  \right)} \right|}{\left|z\right|^N} \times \begin{cases} 2\exp \bigg( {\cfrac{\left| {\nu ^2  - \frac{1}{4}} \right|}{\left|z\right|}} \bigg) & \text{ if } \; \left|\arg z\right| \leq \frac{\pi}{2}, \\ 2\chi\left(N\right) \exp \bigg( { \cfrac{\pi}{2} \cfrac{\left| {\nu ^2  - \frac{1}{4}} \right|}{\left|z\right|}} \bigg) & \text{ if } \; \frac{\pi}{2} < \left|\arg z\right| \leq \pi, \\ \cfrac{4\chi\left(N\right)}{\left|\cos\left(\arg z\right)\right|^{N}} \exp \bigg( { \pi \cfrac{\left| {\nu ^2  - \frac{1}{4}} \right|}{\left|z\cos\left(\arg z\right)\right|}} \bigg) & \text{ if } \; \pi < \left|\arg z\right| < \frac{3\pi}{2}\end{cases}
\end{equation}
(cf. \cite[eq. 7.15]{Olver1} or \cite[exer. 13.2, p. 269]{Olver5}). It is immediate from \eqref{eq558}, \eqref{eq662} and \eqref{eq023} that if $\nu$ is real, our bound \eqref{eq24} (in its region of validity, i.e., when $\left|\nu\right|<N+\frac{1}{2}$) is about twice as sharp as \eqref{eq779}. When $\nu$ is complex and satisfies the assumptions made after \eqref{eq483}, our bound \eqref{eq01} is again better than \eqref{eq779} for large $N$. In particular, if the asymptotic expansion \eqref{eq88} is truncated at or near its numerically least term (i.e., when $N \approx 2\left| z \right|$) and $\nu$ satisfies $\nu ^2  = o\left( {\left| z \right|} \right)$, then the estimate \eqref{eq01} is sharper than \eqref{eq779} for large $z$ (or, equivalently, for large $N$). In any other situations, the estimates \eqref{eq779}, and hence Olver's \eqref{eq776} and \eqref{eq777}, are preferable over \eqref{eq24} or \eqref{eq01}.

In Figures \ref{fig1} and \ref{fig2}, we compare numerically Olver's estimate \eqref{eq779} with the following bound which follows by combining Theorem \ref{theorem6} with Propositions \ref{propb1} and \ref{propb3}:
\begin{equation}\label{eq780}
\big| {R_N^{\left( K \right)} \left( {z,\nu } \right)} \big| \le \frac{{\left| {\cos \left( {\pi \nu } \right)} \right|}}{{\left| {\cos \left( {\pi \Re \left( \nu  \right)} \right)} \right|}}\frac{{\left| {a_N \left( {\Re \left( \nu  \right)} \right)} \right|}}{{\left| z \right|^N }} \times \begin{cases} 1 & \text{ if } \; \left|\arg z\right| \leq \frac{\pi}{2}, \\ \min \left( {\left|\csc \left( {\arg z} \right)\right|,1+\chi \left( {N + \frac{1}{2}} \right) } \right) & \text{ if } \; \frac{\pi}{2} < \left|\arg z\right| \leq \pi.\end{cases}
\end{equation}
The numerical comparison is made for the specific values of $\nu = 2+i$, $\left|z\right|=15$, $0 \leq \arg z \leq \pi$ and with the two different values $N=10$ and $N=30$. Figure \ref{fig2} also illustrates the significance of the factor $\min \left( {\left| {\csc \left( {\arg z} \right)} \right|,1 + \chi \left( {N + \frac{1}{2}} \right)} \right)$ at the optimal truncation $N \approx 2 \left|z\right|$.

It would also be interesting to find bounds for the remainder terms $R_N^{(K')} \left( {z,\nu } \right)$ and $R_N^{(J')} \left( {z,\nu } \right)$ analogous to those \eqref{eq776} and \eqref{eq777}. We leave this problem for future research.

\begin{figure}[!t]
\centering
\def\svgwidth{0.6\textwidth}
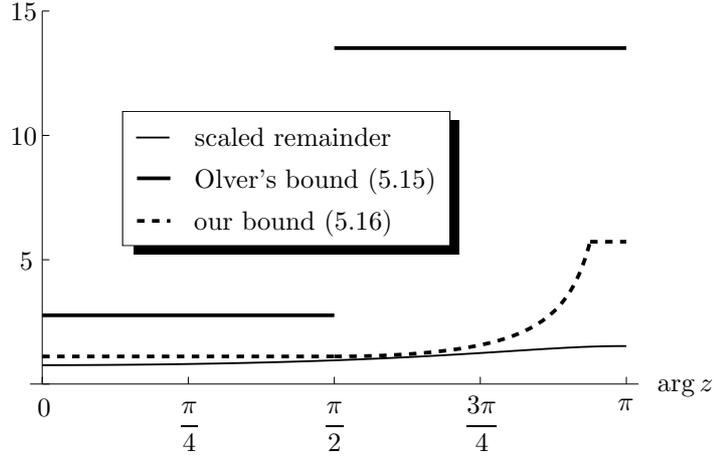
\setlength{\abovecaptionskip}{15pt}\setlength{\belowcaptionskip}{0pt}
\caption{Numerical comparison of different bounds for the scaled remainder term $\big| {R_N^{\left( K \right)} \left( {z,\nu } \right)} \big| / \frac{\left| {a_N \left( \nu  \right)} \right|}{\left|z\right|^N}$ with $N=10$, $\nu=2+i$, $|z|=15$ and $0\leq \arg z\leq \pi$.}
\label{fig1}
\end{figure}

\begin{figure}[!t]
\centering
\def\svgwidth{0.6\textwidth}
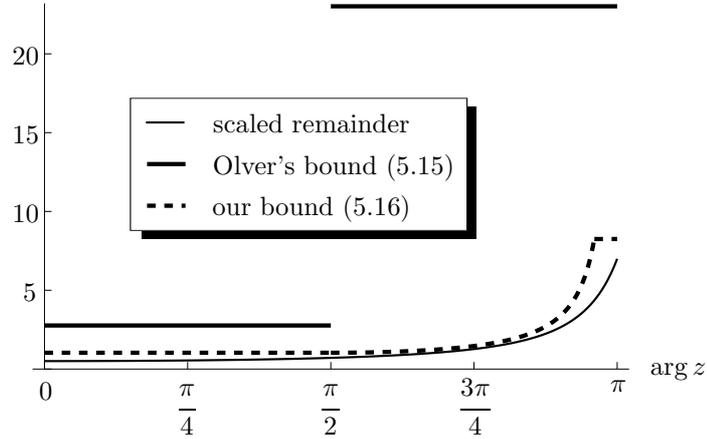
\setlength{\abovecaptionskip}{15pt}\setlength{\belowcaptionskip}{0pt}
\caption{Numerical comparison of different bounds for the scaled remainder term $\big| {R_N^{\left( K \right)} \left( {z,\nu } \right)} \big| / \frac{\left| {a_N \left( \nu  \right)} \right|}{\left|z\right|^N}$ with $N=30$, $\nu=2+i$, $|z|=15$ and $0\leq \arg z\leq \pi$.}
\label{fig2}
\end{figure}

\section*{Acknowledgement} 

The author's research was supported by a research grant (GRANT11863412/70NANB15H221) from the National Institute of Standards and Technology. The author greatly appreciates the help of Dorottya Szir\'{a}ki in improving the presentation of the paper. The author thanks the anonymous referees for their helpful comments and suggestions on the manuscript.

\appendix

\section{Relations between the various remainder terms}\label{appendixa}

In this appendix, we show how the remainder terms of the asymptotic expansions of the functions $H_\nu ^{\left( 1,2 \right)} \left( z \right)$, $H_\nu ^{\left( 1,2 \right)\prime} \left( z \right)$, $Y_\nu \left( z \right)$, $Y'_\nu\left( z \right)$, $I_\nu\left( z \right)$ and $I'_\nu\left( z \right)$ may be expressed in terms of the remainders $R_N^{\left( K \right)} \left( {z,\nu } \right)$, $R_N^{\left( J \right)} \left( {z,\nu } \right)$, $R_N^{(K')} \left( {z,\nu } \right)$ and $R_{N}^{(J')} \left( {z,\nu } \right)$ of the asymptotic expansions of $K_\nu \left( z \right)$, $J_\nu \left( z \right)$, $K'_\nu \left( z \right)$ and $J'_\nu \left( z \right)$.

Let us consider first the Hankel functions. These functions are directly related to the modified Bessel function $K_\nu \left( z \right)$ through the connection formulae
\[
H_\nu ^{\left( 1 \right)} \left( z \right) = \frac{2}{{\pi i}}e^{ - \frac{\pi }{2}i\nu } K_\nu  ( ze^{ - \frac{\pi }{2}i} ),\quad  - \frac{\pi }{2} \le \arg z \le \pi
\]
and
\[
H_\nu ^{\left( 2 \right)} \left( z \right) =  - \frac{2}{{\pi i}}e^{\frac{\pi }{2}i\nu } K_\nu  ( ze^{\frac{\pi }{2}i} ),\quad  - \pi  \le \arg z \le \frac{\pi }{2}
\]
(see, e.g., \cite[10.27.E8]{NIST}). We substitute \eqref{eq70} into the right-hand sides and match the notation with those of \eqref{eq94} and \eqref{eq95} in order to obtain
\begin{equation}\label{eq055}
H_\nu ^{\left( 1 \right)} \left( z \right) = \left( {\frac{2}{{\pi z}}} \right)^{\frac{1}{2}} e^{i\omega } \left( {\sum\limits_{n = 0}^{N - 1} {i^n \frac{{a_n \left( \nu  \right)}}{{z^n }}}  + R_N^{\left( K \right)} (ze^{ - \frac{\pi }{2}i} ,\nu )} \right)
\end{equation}
and
\begin{equation}\label{eq056}
H_\nu ^{\left( 2 \right)} \left( z \right) = \left( {\frac{2}{{\pi z}}} \right)^{\frac{1}{2}} e^{ - i\omega } \left( {\sum\limits_{n = 0}^{N - 1} {\left( { - i} \right)^n \frac{{a_n \left( \nu  \right)}}{{z^n }}}  + R_N^{\left( K \right)} (ze^{\frac{\pi }{2}i} ,\nu )} \right).
\end{equation}
The restrictions on $\arg z$ may now be removed by appealing to analytic continuation.

The corresponding expressions for the derivatives can most readily be obtained by first substituting \eqref{eq055} and \eqref{eq056} into the right-hand sides of the connection formulae $-2H_\nu^{\left(1\right)\prime}\left(z\right) = H_{\nu+1}^{\left(1\right)}\left(z\right)-H_{\nu-1}^{\left(1\right)}\left(z\right)$ and $-2H_\nu^{\left(2\right)\prime}\left(z\right) = H_{\nu+1}^{\left(2\right)}\left(z\right)-H_{\nu-1}^{\left(2\right)}\left(z\right)$ (cf. \cite[10.6.E1]{NIST}). Then employing the relations \eqref{eq101} and \eqref{eq688}, and taking into account the notation of \eqref{eq060} and \eqref{eq061}, we deduce
\[
H_\nu ^{\left( 1 \right)\prime} \left( z \right) = i\left( {\frac{2}{{\pi z}}} \right)^{\frac{1}{2}} e^{i\omega } \left( {\sum\limits_{n = 0}^{N - 1} {i^n \frac{{b_n \left( \nu  \right)}}{{z^n }}}  + R_N^{(K')} ( ze^{ - \frac{\pi}{2}i} ,\nu )} \right)
\]
and
\[
H_\nu ^{\left( 2 \right)\prime} \left( z \right) =  - i\left( {\frac{2}{{\pi z}}} \right)^{\frac{1}{2}} e^{ - i\omega } \left( {\sum\limits_{n = 0}^{N - 1} {\left( { - i} \right)^n \frac{{b_n \left( \nu  \right)}}{{z^n }}}  + R_N^{(K')} ( ze^{\frac{\pi}{2}i} ,\nu )} \right).
\]

To identify the remainder terms in the asymptotic expansion \eqref{eq80} of the Bessel function $Y_\nu  \left( z \right)$, we may proceed as follows. This function is related to the modified Bessel function $K_\nu  \left( z \right)$ via the connection formula
\begin{align*}
-\pi Y_\nu  \left( z \right) = \; & e^{ - \frac{\pi }{2}i\nu } K_\nu  \left( {ze^{ - \frac{\pi }{2}i} } \right) + e^{\frac{\pi }{2}i\nu } K_\nu  \left( {ze^{\frac{\pi }{2}i} } \right)\\
= \; & i\sin \omega \left( {e^{\frac{\pi }{4}i} e^{iz} K_\nu  \left( {ze^{\frac{\pi }{2}i} } \right) + e^{ - \frac{\pi }{4}i} e^{ - iz} K_\nu  \left( {ze^{ - \frac{\pi }{2}i} } \right)} \right) \\ & - \cos \omega \left( {e^{\frac{\pi }{4}i} e^{iz} K_\nu  \left( {ze^{\frac{\pi }{2}i} } \right) - e^{ - \frac{\pi }{4}i} e^{ - iz} K_\nu  \left( {ze^{ - \frac{\pi }{2}i} } \right)} \right),
\end{align*}
with $\left| {\arg z} \right| \leq \frac{\pi }{2}$ \cite[10.27.E10]{NIST}. If we now compare this relation with \eqref{eq080} and \eqref{eq081}, we readily establish that
\begin{gather}\label{eq488}
\begin{split}
Y_\nu  \left( z \right) & = \left( {\frac{2}{{\pi z}}} \right)^{\frac{1}{2}} \left( \sin \omega \left( {\sum\limits_{n = 0}^{N - 1} {\left( { - 1} \right)^n \frac{{a_{2n} \left( \nu  \right)}}{{z^{2n} }}}  + R_{2N}^{\left( J \right)} \left( {z,\nu } \right)} \right) \right. \\  & \hspace{130pt} \left.+ \cos \omega \left( {\sum\limits_{m = 0}^{M - 1} {\left( { - 1} \right)^m \frac{{a_{2m + 1} \left( \nu  \right)}}{{z^{2m + 1} }}}  - R_{2M + 1}^{\left( J \right)} \left( {z,\nu } \right)} \right) \right).
\end{split}
\end{gather}
We can now remove the restriction on $\arg z$ using analytic continuation.

Consider now the derivative $Y'_\nu  \left( z \right)$. The simplest way to derive the required expression is by substituting the formula \eqref{eq488} into the connection formula $-2Y'_\nu\left(z\right) = Y_{\nu+1}\left(z\right)-Y_{\nu-1}\left(z\right)$ (cf. \cite[10.6.E1]{NIST}). Employing the relations \eqref{eq102} and \eqref{eq688}, and matching our notation with that of \eqref{eq108}, we find
\begin{multline*}
Y'_\nu  \left( z \right) = \left( {\frac{2}{{\pi z}}} \right)^{\frac{1}{2}} \left( \cos \omega \left( {\sum\limits_{n = 0}^{N - 1} {\left( { - 1} \right)^n \frac{{b_{2n} \left( \nu  \right)}}{{z^{2n} }}}  + R_{2N}^{(J')} \left( {z,\nu } \right)} \right) \right. \\ \left.-\sin \omega \left( {\sum\limits_{m = 0}^{M - 1} {\left( { - 1} \right)^m \frac{{b_{2m + 1} \left( \nu  \right)}}{{z^{2m + 1} }}}  - R_{2M + 1}^{(J')} \left( {z,\nu } \right)} \right) \right).
\end{multline*}

The modified Bessel function $I_\nu \left( z \right)$ can be treated in the following manner. This function is related to the modified Bessel function $K_\nu  \left( z \right)$ through the connection formula
\begin{equation}\label{eq094}
I_\nu  \left( z \right) =  \mp \frac{i}{\pi }K_\nu  \left( {ze^{ \mp \pi i} } \right) \pm \frac{i}{\pi }e^{ \pm \pi i\nu } K_\nu  \left( z \right)
\end{equation}
(see, for instance, \cite[10.34.E3]{NIST}). Let $N$ and $M$ be arbitrary non-negative integers. We express the functions $K_\nu  \left( {ze^{ \mp \pi i} } \right)$ and $K_\nu  \left( z \right)$ as
\[
K_\nu  \left( {ze^{ \mp \pi i} } \right) =  \pm i\left( {\frac{\pi }{{2z}}} \right)^{\frac{1}{2}} e^z \left( {\sum\limits_{n = 0}^{N - 1} {\left( { - 1} \right)^n \frac{{a_n \left( \nu  \right)}}{{z^n }}}  + R_N^{\left( K \right)} \left( {ze^{ \mp \pi i} ,\nu } \right)} \right)
\]
and
\[
K_\nu  \left( z \right) = \left( {\frac{\pi }{{2z}}} \right)^{\frac{1}{2}} e^{ - z} \left( {\sum\limits_{m = 0}^{M - 1} {\frac{{a_m \left( \nu  \right)}}{{z^m }}}  + R_M^{\left( K \right)} \left( {z,\nu } \right)} \right)
\]
(cf. equation \eqref{eq70}), substitute these expressions into the right-hand side of the functional relation \eqref{eq094} and match the notation with that of \eqref{eq87} in order to obtain
\begin{equation}\label{eq087}
I_\nu  \left( z \right) = \frac{{e^z }}{{\left( {2\pi z} \right)^{\frac{1}{2}} }}\left( {\sum\limits_{n = 0}^{N - 1} {\left( { - 1} \right)^n \frac{{a_n \left( \nu  \right)}}{{z^n }}}  + R_N^{\left( K \right)} \left( {ze^{ \mp \pi i} ,\nu } \right)} \right) \pm ie^{ \pm \pi i\nu } \frac{{e^{ - z} }}{{\left( {2\pi z} \right)^{\frac{1}{2}} }}\left( {\sum\limits_{m = 0}^{M - 1} {\frac{{a_m \left( \nu  \right)}}{{z^m }}}  + R_M^{\left( K \right)} \left( {z,\nu } \right)} \right).
\end{equation}

The analogous expression for the derivative can be deduced by substituting \eqref{eq087} into the right-hand side of the connection formula $2I'_\nu  \left( z \right) = I_{\nu  + 1} \left( z \right) + I_{\nu  - 1} \left( z \right)$ (see, e.g., \cite[10.29.E1]{NIST}). By making use of the relations \eqref{eq101} and \eqref{eq688} and taking into account the notation of \eqref{eq223}, we obtain
\[
I'_\nu  \left( z \right) = \frac{{e^z }}{{\left( {2\pi z} \right)^{\frac{1}{2}} }}\left( {\sum\limits_{n = 0}^{N - 1} {\left( { - 1} \right)^n \frac{{b_n \left( \nu  \right)}}{{z^n }}}  + R_N^{(K')} \left( {ze^{ \mp \pi i} ,\nu } \right)} \right) \mp ie^{ \pm \pi i\nu } \frac{{e^{ - z} }}{{\left( {2\pi z} \right)^{\frac{1}{2}} }}\left( {\sum\limits_{m = 0}^{M - 1} {\frac{{b_m \left( \nu  \right)}}{{z^m }}}  + R_M^{(K')} \left( {z,\nu } \right)} \right).
\]

\section{Bounds for the basic terminants}\label{appendixb}

In this appendix, we prove some simple estimates for the absolute value of the basic terminants $\Lambda _p \left( w \right)$ and $\Pi_p \left( w \right)$ with $p>0$. These estimates depend only on $p$ and the argument of $w$ and therefore also provide bounds for the quantities $\mathop {\sup }\nolimits_{r \geq 1} \big| {\Lambda _{p} \left( {2zr } \right)} \big|$ and $\mathop {\sup }\nolimits_{r \geq 1} \big| {\Pi _{p} \left( {2zr } \right)} \big|$ which appear in Theorems \ref{theorem3}--\ref{theorem7}.

\begin{proposition}\label{propb1} For any $p>0$, the following holds:
\begin{equation}\label{eq368}
\left| {\Lambda _p \left( w \right)} \right| \le \begin{cases} 1 & \text{ if } \; \left|\arg w\right| \leq \frac{\pi}{2}, \\ \min \left( {\left|\csc \left( {\arg w} \right)\right|,\sqrt {e\left( {p + \frac{1}{2}} \right)} } \right) & \text{ if } \; \frac{\pi}{2} < \left|\arg w\right| \leq \pi,\end{cases}
\end{equation}
and
\begin{equation}\label{eq369}
\left| {\Pi _p \left( w \right)} \right| \le \begin{cases} 1 & \text{ if } \; \left|\arg w\right| \leq \frac{\pi}{4}, \\ \min \left( {\left|\csc \left( {2\arg w} \right)\right|,\sqrt {\frac{e}{4}\left( {p + \frac{3}{2}} \right)} } \right) & \text{ if } \; \frac{\pi}{4} < \left|\arg w\right| \leq \frac{\pi}{2}. \end{cases}
\end{equation}
Moreover, when $w$ is positive, we have $0<\Lambda _p \left( w \right)<1$ and $0<\Pi_p \left( w \right)<1$.
\end{proposition}

\begin{proof} To prove the estimate \eqref{eq368}, it suffices to consider the range $0 \le \arg w \le \pi$, because $\Lambda _p \left( {\bar w} \right) = \overline {\Lambda _p \left( w \right)}$. Our starting point is the integral representation
\begin{equation}\label{eq355}
\Lambda _p \left( w \right) = \frac{1}{{\Gamma \left( {p } \right)}}\int_0^{ + \infty } {\frac{{t^{p-1} e^{ - t} }}{{1 + t/w}}dt} ,
\end{equation}
which is valid when $\left|\arg w\right|<\pi$ and $p>0$ (cf. \cite[8.6.E4]{NIST}). For $t \geq 0$, we have
\begin{equation}\label{eq366}
\left| {1 + \frac{t}{w}} \right| \ge \begin{cases} 1 & \text{ if } \; 0 \leq \arg w \leq \frac{\pi}{2}, \\ \sin \left(\arg w\right) & \text{ if } \; \frac{\pi}{2} < \arg w < \pi,\end{cases}
\end{equation}
and therefore
\[
\left| {\Lambda _p \left( w \right)} \right| \le \begin{cases} 1 & \text{ if } \; 0 \leq \arg w \leq \frac{\pi}{2}, \\ \csc \left( {\arg w} \right) & \text{ if } \; \frac{\pi}{2} < \arg w < \pi.\end{cases}
\]
To complete the proof, one has to show that $\left| {\Lambda _p \left( w \right)} \right| \le \sqrt {e\left( {p + \frac{1}{2}} \right)}$ when $\frac{\pi }{2} < \arg w \le \pi$. For this purpose, we deform the contour of integration in \eqref{eq355} by rotating it through an acute angle $\varphi$. Thus, by appealing to Cauchy's theorem and analytic continuation, we have, for arbitrary $0<\varphi<\frac{\pi}{2}$, that
\begin{equation}\label{eq919}
\Lambda _p \left( w \right) = \frac{1}{{\Gamma \left( {p } \right)}}\int_0^{ + \infty e^{i\varphi } } {\frac{{t^{p-1} e^{ - t} }}{{1 + t/w}}dt} 
\end{equation}
when $\frac{\pi }{2} < \arg w \le \pi$. Employing the inequality \eqref{eq366}, we find that
\[
\left| {\Lambda _p \left( w \right)} \right| \le \frac{1}{{\cos ^p \varphi }} \times \begin{cases} 1 & \text{ if } \; \frac{\pi }{2} \le \arg w \le \frac{\pi }{2} + \varphi, \\ \csc \left(\arg w -\varphi \right) & \text{ if } \; \frac{\pi }{2} + \varphi  < \arg w \le \pi.\end{cases}
\]
We now choose the value of $\varphi$ which minimizes the right-hand side of this inequality when $\arg w =\pi$, namely $\varphi  = \arctan (p^{ - \frac{1}{2}} )$. We may then claim that
\[
\left| {\Lambda _p \left( w \right)} \right| \le \frac{1}{{\cos ^p (\arctan (p^{ - \frac{1}{2}} ))}} = \left( {1 + \frac{1}{{p }}} \right)^{\frac{{p }}{2}}  < \sqrt {e\left( {p + \frac{1}{2}} \right)} ,
\]
when $\frac{\pi }{2} \le \arg w \le \frac{\pi }{2} + \arctan (p^{ - \frac{1}{2}} )$, where the last inequality can be obtained by means of elementary analysis. In the remaining case $\frac{\pi }{2} + \arctan (p^{ - \frac{1}{2}} ) < \arg w \le \pi$, we have
\[
\left| {\Lambda _p \left( w \right)} \right| \le \frac{{\csc ( \arg w - \arctan (p^{ - \frac{1}{2}} ))}}{{\cos ^p (\arctan (p^{ - \frac{1}{2}} ))}} \le \frac{{\csc ( \pi  - \arctan (p^{ - \frac{1}{2}} ))}}{{\cos ^p (\arctan (p^{ - \frac{1}{2}} ))}} = \left( {1 + \frac{1}{{p }}} \right)^{\frac{p+1}{2}} p^{\frac{1}{2}}  < \sqrt {e\left( {p + \frac{1}{2}} \right)} .
\]
For the last step, note that the quantity $\left( {1 + \frac{1}{{p}}} \right)^{\frac{p+1}{2}} \sqrt {\frac{{p }}{{p + a}}}$, as a function of $p>0$, increases monotonically if and only if $a\geq \frac{1}{2}$, in which case it has limit $\sqrt{e}$.

The estimate \eqref{eq369} can be proved in a similar way, starting from the representation
\begin{equation}\label{eq356}
\Pi _p \left( w \right) = \frac{1}{{\Gamma \left( {p } \right)}}\int_0^{ + \infty } {\frac{{t^{p-1} e^{ - t} }}{{1 + \left( {t/w} \right)^2 }}dt} ,
\end{equation}
which is valid when $\left|\arg w\right|<\frac{\pi}{2}$ and $p>0$. This representation can be obtained from \eqref{eq355} and the definition of the basic terminant $\Pi _p \left( w \right)$.

Finally, in the case of a positive $w$, notice that $0 < 1/\left(1 + t/w\right) < 1$ and $0 < 1/(1 + \left( {t/w} \right)^2)  < 1$ for any $t>0$. Therefore, the integral representations \eqref{eq355} and \eqref{eq356} combined with the mean value theorem of integration imply that $0<\Lambda _p \left( w \right)<1$ and $0<\Pi_p \left( w \right)<1$.
\end{proof}

\begin{proposition}\label{propb2} For any $p>0$ and $w$ with $\frac{\pi}{2} < \left| {\arg w} \right| < \frac{3\pi}{2}$, we have
\begin{equation}\label{eq917}
\left| {\Lambda _p \left( w \right)} \right| \le \frac{{\left| {\csc \left( {\arg w - \varphi } \right)} \right|}}{{\cos ^p \varphi }},
\end{equation}
where $\varphi$ is the unique solution of the implicit equation
\[
\left( {p + 1} \right)\cos \left( {\arg w -2\varphi} \right) = \left( {p - 1} \right)\cos\left(\arg w\right)
\]
that satisfies $0 < \varphi  <  - \frac{\pi }{2} + \arg w$ if $\frac{\pi }{2} < \arg w < \pi$, $ - \pi  + \arg w  < \varphi < \frac{\pi }{2}$ if $\pi  \le \arg w  < \frac{3\pi}{2}$, $0 < \varphi < \frac{\pi }{2} + \arg w$ if $- \pi  < \arg w <  -\frac{\pi}{2}$ and $ - \frac{\pi}{2} < \varphi < \pi  + \arg w$ if $- \frac{{3\pi }}{2} < \arg w  \le  - \pi $.

Similarly, for any $p>0$ and $w$ with $\frac{\pi}{4} < \left| {\arg w} \right| < \pi$, we have
\begin{equation}\label{eq918}
\left| {\Pi _p \left( w \right)} \right| \le \frac{{\left| {\csc \left( {2\left( {\arg w - \varphi' } \right)} \right)} \right|}}{{\cos ^p \varphi' }},
\end{equation}
where $\varphi'$ is the unique solution of the implicit equation
\[
\left( {p + 2} \right)\cos \left( {2\arg w -3\varphi'} \right) = \left( {p - 2} \right)\cos\left(2\arg w-\varphi'\right)
\]
that satisfies $0 < \varphi'  <  - \frac{\pi}{4} + \arg w$ if $\frac{\pi}{4} < \arg w  < \frac{\pi}{2}$, $ - \frac{\pi}{2}  + \arg w  < \varphi'  < -\frac{\pi}{4}+\arg w$ if $\frac{\pi}{2}  \le \arg w  < \frac{3\pi}{4}$, $ - \frac{\pi}{2}  + \arg w  < \varphi'  < \frac{\pi }{2}$ if $\frac{3\pi}{4}  \le \arg w  < \pi$, $\frac{\pi}{4} + \arg w < \varphi'  <  0$ if $-\frac{\pi }{2} < \arg w  < -\frac{\pi}{4}$, $\frac{\pi}{4}  + \arg w  < \varphi'  < \frac{\pi}{2}+\arg w$ if $-\frac{3\pi}{4}  < \arg w  \le -\frac{\pi}{2}$ and $ - \frac{\pi}{2}  < \varphi'  < \frac{\pi }{2}+ \arg w$ if $-\pi < \arg w \le -\frac{3\pi}{4}$.
\end{proposition}

We remark that the values of $\varphi$ and $\varphi'$ in this proposition are chosen so as to minimize the right-hand sides of the inequalities \eqref{eq917} and \eqref{eq918}, respectively.

\begin{proof} It is enough to prove \eqref{eq917} when $\frac{\pi}{2}<\arg w< \frac{3\pi}{2}$, because this inequality for the case $-\frac{3\pi}{2}<\arg w< -\frac{\pi}{2}$ can be derived by taking complex conjugates. Note that the representation \eqref{eq919} is actually valid in the wider range $\frac{\pi }{2} < \arg w < \pi  + \varphi$. Thus, we can infer that
\[
\left| {\Lambda _p \left( w \right)} \right| \le \frac{{\csc \left( {\arg w - \varphi } \right)}}{{\cos^p \varphi }},
\]
provided $\frac{\pi }{2}+\varphi < \arg w < \pi  + \varphi$. We would like to choose $\varphi$ in a way that the right-hand side of this inequality is minimized. A lemma of Meijer \cite[pp. 953--954]{Meijer} shows that this minimization problem has a unique solution with the properties given in the statement of Proposition \ref{propb2}.

The bound for $\Pi _p \left( w \right)$ can be proved similarly: we deform the contour of integration in \eqref{eq356} by rotating it through an arbitrary acute angle $\varphi'$, and then we employ the inequality \eqref{eq366} and use the corresponding lemma of Meijer \cite[p. 956]{Meijer} to minimize the resulting estimate.
\end{proof}

\begin{proposition}\label{propb3} For any $p>0$, we have
\begin{equation}\label{eq811}
\left| {\Lambda _p \left( w \right)} \right| \le 1 + \Gamma \left( {\frac{{p }}{2} + 1} \right){\bf F}\left( {\frac{1}{2},\frac{p}{2};\frac{{p}}{2} + 1;\cos ^2 \left( {\arg w} \right)} \right) \le 1 + \chi \left( {p } \right),
\end{equation}
for $\frac{\pi }{2} < \left| {\arg w} \right| \le \pi$, and
\begin{equation}\label{eq812}
\left| {\Pi _p \left( w \right)} \right| \le 1 + \frac{1}{2}\Gamma \left( {\frac{{p }}{2} + 1} \right){\bf F}\left( {\frac{1}{2},\frac{p}{2};\frac{{p }}{2} + 1;\sin ^2 \left( {\arg w} \right)} \right) \le 1 + \frac{{\chi \left( {p } \right)}}{2},
\end{equation}
for $\frac{\pi }{4} < \left| {\arg w} \right| \le \frac{\pi }{2}$. Here $\chi \left( {p } \right)$ is defined by \eqref{eq611}.
\end{proposition}

\begin{proof} It is sufficient to prove \eqref{eq811} for $\frac{\pi}{2}<\arg w\leq \pi$, as the estimates for $-\pi\leq\arg w<-\frac{\pi}{2}$ can be derived by taking complex conjugates. To do so, we consider the integral representation
\[
\Lambda _p \left( w \right) = \int_0^{ + \infty } {\frac{{e^{ - t} }}{{\left( {1 + t/w} \right)^p }}dt} ,
\]
which is valid when $\left|\arg w\right|<\pi$ and $p>0$ (cf. \cite[8.6.E5]{NIST}). Integrating once by parts, we obtain
\begin{equation}\label{eq433}
\Lambda _p \left( w \right) = 1 - \frac{p}{w}\int_0^{ + \infty } {\frac{{e^{ - t} }}{{\left( {1 + t/w} \right)^{p + 1} }}dt} .
\end{equation}
Assuming that $\frac{\pi}{2}<\arg w<\pi$, we can deform the contour of integration in \eqref{eq433} by rotating it through a right angle. And therefore, by Cauchy's theorem and analytic continuation, we have
\begin{equation}\label{372}
\Lambda _p \left( w \right) = 1 - \frac{{p }}{w}\int_0^{ + \infty e^{\frac{\pi }{2}i} } {\frac{{e^{ - t} }}{{\left( {1 + t/w} \right)^{p + 1} }}dt}  = 1 - i\frac{p}{w}\int_0^{ + \infty } {\frac{{e^{ - it} }}{{\left( {1 + it/w} \right)^{p + 1} }}dt}
\end{equation}
when $\frac{\pi}{2}<\arg w\leq \pi$. Hence, we may assert that
\begin{align*}
\left| {\Lambda _p \left( w \right)} \right| & \le 1 + \frac{{p }}{{\left| w \right|}}\int_0^{ + \infty } {\frac{{dt}}{{\left| {1 + it/w} \right|^{p + 1} }}}  = 1 + \frac{{p }}{{\left| w \right|}}\int_0^{ + \infty } {\frac{{dt}}{{(1 + 2t/\left| w \right|\sin \left( {\arg w} \right) + t^2 /\left| w \right|^2 )^{\frac{p+1}{2}} }}} 
\\ & = 1 + p \int_0^{ + \infty } {\frac{{du}}{{(1 + 2u\sin \left( {\arg w} \right) + u^2 )^{\frac{p+1}{2}} }}} .
\end{align*}
To simplify this result, we can proceed as follows. When $\arg w=\pi$, we perform a change of integration variable from $u$ to $t$ by $t=u^2$, and, using the beta integral together with the known evaluation of the regularized hypergeometric function with argument $1$ (see, e.g., \cite[5.12.E3 and 15.4.E20]{NIST}), we obtain
\[
\left| {\Lambda _p \left( w \right)} \right|  \le 1 + \frac{p}{2}\int_0^{ + \infty } {\frac{{t^{ - \frac{1}{2}} }}{{\left( {1 + t} \right)^{\frac{{p + 1}}{2}} }}dt}  = 1 + \frac{{\Gamma \left( {\frac{1}{2}} \right)\Gamma \left( {\frac{p}{2} + 1} \right)}}{{\Gamma \left( 1 \right)\Gamma \left( {\frac{p}{2} + \frac{1}{2}} \right)}} = 1 + \Gamma \left( {\frac{p}{2} + 1} \right){\bf F}\left( {\frac{1}{2},\frac{p}{2};\frac{p}{2} + 1;1} \right).
\]
Since $\cos^2\left(\pi\right)=1$, this result is in agreement with \eqref{eq811}. In the case that $\frac{\pi}{2}<\arg w < \pi$, a change of variable from $u$ to $t$ via $t = 1 - \sin ^2 (\arg w)/(u + \sin (\arg w))^2$ and the well-known integral representation of the regularized hypergeometric function (see \cite[15.6.E1]{NIST}) yield
\begin{multline*}
\left| {\Lambda _p \left( w \right)} \right|  \le 1 + \frac{p}{2}\sin \left( {\arg w} \right)\int_0^1 {\frac{{\left( {1 - t} \right)^{\frac{p}{2} - 1} }}{{(1 - t\cos ^2 \left( {\arg w} \right))^{\frac{{p + 1}}{2}} }}dt} \\ = 1 + \sin \left( {\arg w} \right)\Gamma \left( {\frac{p}{2} + 1} \right){\bf F}\left( {\frac{p}{2}+\frac{1}{2},1;\frac{p}{2} + 1;\cos ^2 \left( {\arg w} \right)} \right).
\end{multline*}
The linear transformation $\sin (\arg w){\bf F}\left( {\frac{p}{2}+\frac{1}{2},1;\frac{p}{2} + 1;\cos ^2 (\arg w)} \right) = {\bf F}\left( {\frac{1}{2},\frac{p}{2};\frac{p}{2} + 1;\cos ^2 (\arg w)} \right)$ (cf. \cite[15.8.E1]{NIST}) then shows that this bound is equivalent to the required one in \eqref{eq811}.

To obtain the second inequality in \eqref{eq811}, note that
\begin{align*}
\Gamma \left( {\frac{{p }}{2} + 1} \right){\bf F}\left( {\frac{{1 }}{2},\frac{p}{2};\frac{{p  }}{2}  + 1;\cos ^2 \left( {\arg w} \right)} \right) & \le \Gamma \left( {\frac{{p }}{2} + 1} \right){\bf F}\left( {\frac{{1 }}{2},\frac{p}{2};\frac{{p}}{2} + 1;1} \right)
\\ & = \frac{{\Gamma \left( {\frac{1}{2}} \right)\Gamma \left( {\frac{{p }}{2} + 1} \right)}}{{\Gamma \left( 1 \right)\Gamma \left( {\frac{p}{2} + \frac{1}{2}} \right)}} = \pi ^{\frac{1}{2}} \frac{{\Gamma \left( {\frac{{p }}{2} + 1} \right)}}{{\Gamma \left( {\frac{{p }}{2} + \frac{1}{2}} \right)}} = \chi \left( {p } \right).
\end{align*}

The bounds \eqref{eq812} for $\Pi _p \left( w \right)$ may be proved as follows. From the definition of the basic terminant $\Pi _p \left( w \right)$, we can infer that
\[
\left| {\Pi _p \left( w \right)} \right| \le \frac{1}{2}\left( {\big| {\Lambda _p (we^{\frac{\pi }{2}i} )} \big| + \big| {\Lambda _p (we^{ - \frac{\pi }{2}i} )} \big|} \right).
\]
When $\frac{\pi }{4} < \pm \arg w \le \frac{\pi }{2}$, the term $\big| {\Lambda _p (we^{\pm\frac{\pi }{2}i} )} \big|$ is bounded by \eqref{eq811}, and we can see that $\big| {\Lambda _p (we^{ \mp \frac{\pi }{2}i} )} \big| \leq 1$ from \eqref{eq368}.
\end{proof}

Before we proceed to our last set of bounds, we would like to compare the estimates given in Propositions \ref{propb1} and \ref{propb3}. For the purpose of brevity, we consider only the bounds for $\Lambda _p \left( w \right)$; the other basic terminant $\Pi_p \left( w \right)$ can be treated in a similar way.

First, assume that $\frac{\pi}{2}<\left|\arg w\right|<\pi$, $\left|\arg w\right|$ is bounded away from $\pi$ and $p$ is large. Employing a linear transformation formula and the large-$c$ asymptotics of the regularized hypergeometric function ${\bf F}\left( {a,b;c;w} \right)$ (see, e.g., \cite[15.8.E1 and 15.12.E2]{NIST}), we find that
\begin{align*}
{\bf F}\left( {\frac{1}{2},\frac{p}{2};\frac{p}{2} + 1;\cos ^2 \left( {\arg w} \right)} \right) & = \left| {\csc \left( {\arg w} \right)} \right|{\bf F}\left( {\frac{1}{2},1;\frac{p}{2} + 1; - \cot ^2 \left( {\arg w} \right)} \right) \\ & = \frac{{\left| {\csc \left( {\arg w} \right)} \right|}}{{\Gamma \left( {\frac{p}{2} + 1} \right)}}\left( {1 + \mathcal{O}_{\arg w} \left( {\frac{1}{p}} \right)} \right).
\end{align*}
Consequently, from \eqref{eq811}, we have
\[
\left| {\Lambda _p \left( w \right)} \right| \le 1 + \left| {\csc \left( {\arg w} \right)} \right|\left( {1 + \mathcal{O}_{\arg w} \left( {\frac{1}{p}} \right)} \right).
\]
By comparing this inequality with \eqref{eq368}, one sees that, under the above circumstances, the estimate \eqref{eq368} is sharper than \eqref{eq811}.

Consider now the case when $\frac{\pi}{2}<\left|\arg w\right|<\pi$, $\left|\arg w\right|$ is close to $\pi$ and $p$ is bounded. A linear transformation formula \cite[15.8.E4]{NIST} and the definition of the regularized hypergeometric function yield
\begin{align*}
{\bf F}\left( {\frac{1}{2},\frac{p}{2};\frac{p}{2} + 1;\cos ^2 \left( {\arg w} \right)} \right)  =\; & \frac{\pi }{{\Gamma \left( {\frac{p}{2} + \frac{1}{2}} \right)}}{\bf F}\left( {\frac{1}{2},\frac{p}{2};\frac{1}{2};\sin ^2 \left( {\arg w} \right)} \right) \\ & - \pi ^{\frac{1}{2}} \frac{{\left| {\sin \left( {\arg w} \right)} \right|}}{{\Gamma \left( {\frac{p}{2}} \right)}}{\bf F}\left( {\frac{p}{2} + \frac{1}{2},1;\frac{3}{2};\sin ^2 \left( {\arg w} \right)} \right)
\\ = \; &  \frac{{\pi ^{\frac{1}{2}} }}{{\Gamma \left( {\frac{p}{2} + \frac{1}{2}} \right)}} - \frac{{2\left| {\sin \left( {\arg w} \right)} \right|}}{{\Gamma \left( {\frac{p}{2}} \right)}} + \mathcal{O}_p \left( {\sin ^2 \left( {\arg w} \right)} \right).
\end{align*}
Whence, by \eqref{eq811}, we deduce
\[
\left| {\Lambda _p \left( w \right)} \right| \le 1 + \chi \left( p \right) - \left| {\sin \left( {\arg w} \right)} \right|p + \mathcal{O}_p \left( {\sin ^2 \left( {\arg w} \right)} \right).
\]
Comparison with \eqref{eq368} then shows that, under the above circumstances, the bound \eqref{eq811} is sharper and behaves better than \eqref{eq368}, unless perhaps when $\sqrt {e\left( {p + \frac{1}{2}} \right)}  < \left| {\csc \left( {\arg w} \right)} \right|$.

Finally, we discuss the case when $\left|\arg w\right|=\pi$. It can readily be shown that
\[
\sqrt {e\left( {p + \frac{1}{2}} \right)}  < 1 + \chi \left( {p } \right) \; \text{ for } \; 0 < p < 5.9564 \ldots ,
\]
whence \eqref{eq368} is sharper than \eqref{eq811} when $\left|\arg w\right|$ is equal or close to $\pi$ and $p$ is small. On the other hand, the saddle point method applied to \eqref{372} shows that when $\left|\arg w\right| = \pi$ and $\left|p - \left| w \right|\right|$ is bounded, the asymptotics
\[
\left| {\Lambda _p \left( w \right)} \right| \sim \left( {\frac{{\pi p}}{2}} \right)^{\frac{1}{2}}
\]
holds as $p\to +\infty$. Thus, using the large-$p$ asymptotics $\chi \left( p \right) \sim \left( {\frac{{\pi p}}{2}} \right)^{\frac{1}{2}}$, we can infer that the inequality $\left| {\Lambda _p \left( w \right)} \right| \le 1 + \chi \left( p \right)$ is asymptotically sharp.

\begin{proposition}\label{propb4} For any $p>0$, we have
\begin{equation}\label{eq930}
\left| {\Lambda _p \left( w \right)} \right| \le \frac{{\sqrt {2\pi p} }}{{\left| {\cos \left( {\arg w} \right)} \right|^{p} }} + \left| {\Lambda _p (we^{ \mp 2\pi i} )} \right| \le \frac{{2\chi \left( {p  } \right)}}{{\left| {\cos \left( {\arg w} \right)} \right|^{p} }} + \left| {\Lambda _p (we^{ \mp 2\pi i} )} \right|,
\end{equation}
for $\pi < \pm \arg w < \frac{3\pi}{2}$, and
\begin{equation}\label{eq931}
\left| {\Pi _p (w)} \right| \le \frac{{\sqrt {2\pi p} }}{{2\left| {\sin \left( {\arg w} \right)} \right|^{p} }} + \left| {\Pi _p (we^{ \mp \pi i} )} \right| \le \frac{{\chi \left( {p } \right)}}{{\left| {\sin \left( {\arg w} \right)} \right|^{p } }} + \left| {\Pi _p (we^{ \mp \pi i} )} \right|,
\end{equation}
for $\frac{\pi}{2} < \pm \arg w <\pi$.
\end{proposition}

The dependence on $\left|w\right|$ in these estimates may be eliminated by employing the bounds for $\left| {\Lambda _p (we^{ \mp 2\pi i} )} \right|$ and $\left| {\Pi _p (we^{ \mp \pi i} )} \right|$ that were derived previously.

\begin{proof} The proof of \eqref{eq930} is based on the functional relation
\begin{equation}\label{eq555}
\Lambda _p \left( w \right) =  \pm 2\pi i\frac{{e^{ \mp \pi ip} }}{{\Gamma \left( {p } \right)}}w^{p} e^w  + \Lambda _p (we^{ \mp 2\pi i} )
\end{equation}
(cf. \cite[8.2.E10]{NIST}). We take the upper or lower sign in \eqref{eq555} according as $\pi < \arg w < \frac{3\pi}{2}$ or $-\frac{3\pi}{2} < \arg w < -\pi$. Now, from \eqref{eq555} we can infer that
\[
\left| {\Lambda _p \left( w \right)} \right| \le 2\pi \frac{1}{{\Gamma \left( {p } \right)}}\left| w \right|^{p } e^{ - \left| w \right|\left| {\cos \left( {\arg w} \right)} \right|}  + \left| {\Lambda _p (we^{ \mp 2\pi i} )} \right|.
\]
Notice that the quantity $r^{p} e^{ - r\alpha }$, as a function of $r>0$, takes its maximum value at $r=p/\alpha$ when $\alpha>0$ and $p>0$. We therefore find that
\[
\left| {\Lambda _p \left( w \right)} \right| \le 2\pi \frac{1}{{\Gamma \left( {p } \right)}}\frac{{p^p e^{ - p} }}{{\left| {\cos \left( {\arg w} \right)} \right|^{p } }} + \left| {\Lambda _p (we^{ \mp 2\pi i} )} \right| \le \frac{{\sqrt {2\pi p} }}{{\left| {\cos \left( {\arg w} \right)} \right|^{p } }} + \left| {\Lambda _p (we^{ \mp 2\pi i} )} \right|.
\]
The second inequality can be obtained from the well known fact that $\sqrt {2\pi } p^{p - \frac{1}{2}} e^{ - p}  \le \Gamma \left( p \right)$ for any $p>0$ (see, for instance, \cite[5.6.E1]{NIST}). We derive the second bound in \eqref{eq930} from the result that
\[
\left(\frac{p}{2}\right)^{\frac{1}{2}}  \le \frac{{\Gamma \left( {\frac{p}{2} + 1} \right)}}{{\Gamma \left( {\frac{p}{2} + \frac{1}{2}} \right)}}
\]
for any $p>0$ (see, e.g., \cite[5.6.E4]{NIST}) and the definition of $\chi\left(p\right)$.

The estimate \eqref{eq931} can be deduced in an analogous manner, starting from the functional relation
\[
\Pi _p (w) = \pm \pi i \frac{{e^{ \mp \frac{\pi }{2}ip} }}{{\Gamma \left( {p } \right)}}w^{p } e^{ \pm iw}  + \Pi _p (we^{ \mp \pi i} ),
\]
which can be obtained from \eqref{eq555} and the definition of the basic terminant $\Pi _p \left( w \right)$.
\end{proof}

\end{document}